\documentclass[11pt]{article}
\usepackage{amssymb,amsmath,amsthm,cases,color,relsize}
\usepackage[]{graphicx}
\usepackage[]{bm}
\usepackage[]{colonequals}
\usepackage{bbm}

\newcommand{\calD}{\mathcal{D}}
\newcommand{\calF}{\mathcal{F}}
\newcommand{\calG}{\mathcal{G}}
\newcommand{\calH}{\mathcal{H}}

\def\To{\longrightarrow}
\def\eps{\varepsilon}
\def\wh{\widehat}
\def\wt{\widetilde}

\def\inprobto{{\buildrel P\over \To}}
\def\G{\mathcal{G}}
\def\F{\mathcal{F}}
\def\FF{\mathbb F}
\def\GG{\mathbb G}
\def\CC{\mathbb C}
\def\ZZ{\mathbb Z}

\def\RR{\mathbb R}
\def\EE{\mathbb{E}}

\def\PP{\mathbb{P}}

\def\1{\mathbbm{1}}
\def\GG{\mathbb{G}}
\def\UU{\mathbb{U}}

\def\ZZ{\mathbb{Z}}
\def\wh{\widehat}

\def\u{\mathbf{u}}
\def\v{\mathbf{v}}
\def\x{\mathbf{x}}
\def\y{\mathbf{y}}

\def\t{\mathbf{t}}
\def\X{\mathbf{X}}

\def\Y{\mathcal{Y}}

\def\1{\mathbf{1}}
\def\a{\mathbf{a}}
\def\b{\mathbf{b}}
\def\c{\mathbf{c}}
\def\d{\mathrm{d}}
\def\z{\mathbf{z}}

\newtheorem{theorem}{Theorem}
\newtheorem{lemma}[theorem]{Lemma}
\newtheorem{proposition}[theorem]{Proposition}
\newtheorem{cor}[theorem]{Corollary}

\newtheorem{remark}[theorem]{Remark}
\usepackage[margin=1in]{geometry}

\title{Weak convergence of  empirical copula processes indexed by functions}
\author{Dragan Radulovi\'c\\ {\small  Department of Mathematics, Florida Atlantic University}\\
Marten Wegkamp\\ {\small  Department of Mathematics \&
Department of Statistical Science, Cornell University}\\
Yue Zhao\\ {\small Department of Statistical Science, Cornell University}}
\date{\today}

\begin{document}

\maketitle \pagestyle{myheadings} \markboth{}{}

\begin{abstract}
Weak convergence of the empirical copula process indexed by a class of functions is established.
Two scenarios are considered in which either
 some smoothness of these functions or smoothness of the underlying copula function is required.\\
A novel integration by parts formula for multivariate, right continuous functions of bounded variation, which is perhaps of independent interest, is proved.
It is a key ingredient in proving weak convergence of a general empirical process  indexed by functions of bounded variation.\\

\noindent Running title: Weak convergence of empirical copula processes\\

 \noindent
MSC2000 Subject classification: {Primary 60F17 ; secondary 26B20, 26B30, 60G99.}\\


\end{abstract}

\baselineskip=18pt


\baselineskip=18pt

\section{Introduction}
Let $F$ be a distribution function in $\RR^d$ with continuous marginals $F_j$, $j=1,\ldots,d$ and copula function $C$.
Given an i.i.d. sample $\X_1,\ldots,\X_n$ from $F$,
we can construct the empirical distribution function $$\FF_n(\x) = \frac1n \sum_{i=1}^n 1\{ \X_i \le \x\},\  \x\in\RR^d,$$ with marginals
$\FF_{nj}$, $j\in\{1,\ldots,d\}$.
 The empirical copula function is defined by
\begin{eqnarray}
\CC_n(\u) &=& \FF_n( \FF_{n1}^{-} (u_1),\ldots, \FF_{nd}^{-}(u_d) ),\qquad \u=(u_1,\ldots,u_d) \in [0,1]^d
\end{eqnarray}
and the (ordinary) empirical copula process is given by
\begin{eqnarray}
\sqrt{n} (\CC_n -C)(\u), \qquad \u\in[0,1]^d.
\end{eqnarray}
 Weak convergence of the empirical copula process is well studied, see Stute~(1984), G\"anssler \& Stute~(1987), Fermanian et al.~(2004). Segers~(2012) obtained weak convergence under the weak condition that  the first-order partial derivatives of the copula $C$ exist and are continuous on the interior of the unit hypercube. He  slightly relaxed the condition used in Fermanian et al.~(2004) that required existence and  continuity of the first-order partial derivatives of $C$ on the {\em entire} hypercube.
This  is a sharp  condition as Theorem~4 of Fermanian et al.~(2004) shows that the empirical copula process no longer converges if the continuity of any of the $d$ first-order partial derivatives fails at a point  $\u\in (0,1)^d$.  
B\"ucher, Segers and Vogulshev (2014) use a weaker semi-metric on $\ell^\infty([0,1]^d)$ and obtain {\em hypi}-convergence of the empirical copula process,
under the condition that the set of points in $[0,1]^d$ where the partial derivatives of the copula $C$ exist and are continuous has Lebegue measure one.
They show that  this  convergence  still implies  weak convergence of certain Cram\'er-von Mises test statistics. \\

While it can be verified that  $\CC_n$ is left-continuous with right-hand limits, its cousin
\begin{eqnarray}
\bar \CC_n(\u) &=& \frac{1}{n} \sum_{i=1}^n 1\{ \FF_{n1}(X_{i1})\le u_1,\ldots,\FF_{nd}(X_{id})\le u_d \}, \qquad \u=(u_1,\ldots,u_d)\in [0,1]^d
\end{eqnarray}
is  {\em c\`adl\`ag} (right-continuous with left-hand limits) and as such a more standard object in probability  theory  and, in particular,  Lebesgue-Stieltjes integration. The empirical copula processes $\sqrt{n} (\CC_n -C)(\u)$ and $\sqrt{n} (\bar \CC_n -C)(\u)$ are asymptotically equivalent as
\begin{eqnarray}
\sup_{\u\in[0,1]^d} \left|\sqrt{n} (\CC_n -C)(\u)-\sqrt{n} (\bar \CC_n -C)(\u) \right| \le \frac{2}{\sqrt{n}},
\end{eqnarray}
as pointed out by  Fermanian et al.~(2004, proof of Theorem~6), and hence the process
\begin{eqnarray}\label{wc:bar C}
 \sqrt{n} (\bar \CC_n -C)(\u) , \qquad \u\in[0,1]^d
\end{eqnarray}
converges weakly in $\ell^\infty([0,1]^d)$ under the same weak assumptions as in Segers~(2012).\\

This paper addresses the following question: {\em Can we generalize the empirical copula process to a process indexed by functions on the unit hypercube, rather than points in the unit hypercube? }
We consider the  generalization
\begin{eqnarray}\label{bar Z_n}
\bar \ZZ_n(g) &=&\sqrt{n} \int g\,{\rm d} (\bar \CC_n- C)\\
&=& \frac{1}{\sqrt{n}} \sum_{i=1}^n \left\{  g(\FF_{n1}(X_{i1}),\ldots, \FF_{nd}(X_{id})) - \EE [ g(F_1(X_{i1}),\ldots,F_d(X_{id}))] \right\}\nonumber
\end{eqnarray}
based on the  {\em c\`adl\`ag} version $\bar \CC_n$ of $\CC_n$.
This generalization is of particular interest because  $\bar \ZZ_n(g)$ is a multivariate rank order statistic and   common in the statistics  literature.
See Ruymgaart et al~(1972), Ruymgaart~(1974) and R\"uschendorf~(1976) for early references.
For this reason, we take $\bar \CC_n$ as our starting point.  Clearly, (\ref{bar Z_n}) reduces to (\ref{wc:bar C}) for $g(\v)=1\{\v\le \u\}$, and
Theorem~6 in Fermanian et al (2004) states that the statistic (\ref{bar Z_n}) has a normal limit distribution under suitably regular functions  $g: [0,1]^d \to\RR$.  This leads to the question ``{\em Can we characterize the class $\G$  of functions $g: [0,1]^d \to\RR$ for which (\ref{bar Z_n}) converges weakly in $\ell^\infty(\G)$?}''\\

To answer this question, we consider two complementary cases,
one that requires some smoothness of the underlying copula $C$ and one that requires smoothness of the indexing functions $g\in\G$.
Van der Vaart \& Wellner (2007) showed that if the functions $g$ are sufficiently  smooth, then existence of first-order partial derivatives of $C$ is no longer required for the weak convergence of $\bar \ZZ_n$.
This remarkable fact  was established in Corollary 5.4 of Van der Vaart \& Wellner (2007).  Theorem~2 corrects a minor  mistake in their proof (uniform equicontinuity in lieu of mere continuity of the partial derivatives is required)  and demonstrates  the weak convergence in a different way under weaker conditions on  $\G$ that require no explicit entropy conditions on $\G$.
We stress that many well-known copulas are not differentiable, for example,  the Frechet-Hoeffding copulas,
the Marshal-Olkin copula, the Cuadras-Aug\'e copula, the Raftery copula, among many others,
see the monograph by Nelsen (1999).
Moreover, many of the common goodness-of-fit tests for copulas rely
on the weak convergence of the standard copula process and thus do not apply
in   non-differentiable settings.\\

The scenario where $C$ is sufficiently smooth, while functions in $\G$ are not necessarily differentiable  has not been addressed in the literature. In case the underlying copula satisfies  Segers (2012) condition, we show that under mild conditions on $\G$ the process $\bar \ZZ_n(g)$ converges weakly. We found a surprisingly simple proof for this fact based on  the very general result, Theorem~\ref{thm:GENERAL} below. This theorem  is of interest in its own, and it is essentially the non-trivial  $d$-dimensional version of an integration by parts trick introduced in Radulovi\'c and Wegkamp (2015).\\

The paper is organized as follows. Section 2 presents a general weak convergence result of  empirical processes, indexed by functions of bounded variation, including empirical processes  based on   stationary sequences satisfying strong alpha-mixing conditions.
We stress that alpha-mixing is the least restrictive form of available
mixing assumptions in the literature. The  few results  that treat empirical processes  
indexed by functions $g\in\G$,
 all require  stringent conditions on the entropy numbers of $\G$ and on
the rate of decay for the mixing coefficients of $\X_i$, see, e.g., Andrews and Pollard (1994). The main culprit is  that alpha-mixing does not allow for sharp exponential inequalities for partial
sums. The only known cases for which sharp conditions do exist are under
more restrictive, beta-mixing dependence. The latter  allows for decoupling and
  yields exponential inequalities not unlike the i.i.d. case (Arcones \& Yu (1994), Doukhan, Massart \& Rio (1995)).
Our theory does not stop there and allows for  for short memory casual linear sequences (Doukhan
\& Surgailis, 1998). This work extends Radulovic \& Wegkamp (2015) to the multidimensional case. 
Dehling et al. (2009)   prove  weak convergence  of the standard empirical processes based a stationary
sequences that are not necessarily mixing.  Dehling et al. (2014) treat more general processes indexed by
classes of functions under  cumbersome entropy conditions on $\G$.  The advantage of the method presented in this paper is that no explicit entropy condition  on the set $\G$ is imposed, while only weak convergence of the standard empirical process is required.\\
Section 3 presents the main results for empirical copula processes indexed by functions. Smoothness  of either the copula function  $C$ or the indexing functions is required.\\
The proofs of the results in Section 3 are collected in Section 4.\\
 Finally, the appendix  contains a novel integration by parts formula for multivariate, right continuous functions of bounded variation, which is perhaps of independent interest.\\

\subsection{Notations}
\label{sec:notations}

We list in this subsection the notations necessary to address the multivariate extension of the concept of bounded variation and the integration by parts formula in this paper.  We mostly follow the notations introduced in Owen~(2005,~Section~3).  For $\x\in\mathbb{R}^d$, we denote its $j$th component as $x_j$, that is, $\x=(x_1,\dots,x_d)$.  We let ${\bf 0}\in\mathbb{R}^d$ be the vector with all components equal to zero, and $\1\in\mathbb{R}^d$ be the vector with all components equal to one.  For $\a,\b\in\mathbb{R}^d$, we write $\a<\b$ or $\a\le\b$ if these inequalities hold for all $d$ components.  For $\a,\b\in\mathbb{R}^d$ with $\a\le\b$, the hypercube $[\a,\b]$ is the set $\{\x\in\mathbb{R}^d:\a\le\x\le\b\}$.  Thus $[{\bf 0},\1]=[0,1]^d$ is the closed unit hypercube, and in this paper we will work exclusively over this domain unless specified otherwise.  Also $(\a,\b) = \{\x\in\mathbb{R}^d:\a<\x<\b\}$ and $[\a,\b)$ and $(\a,\b]$ are defined similarly.

For $I,J\subset\{1,\dots,d\}$, we write $|I|$ for the cardinality of $I$, and $I-J$ for the complement of $J$ with respect to $I$.  A unary minus denotes the complement with respect to $\{1,\dots,d\}$, so that $-I= \{1,\dots,d\}-I$.  In expressions involving both the unary minus and other set operations, the unary minus has the highest precedence; for instance, $-I-J=(\{1,\dots,d\}-I)-J$.

For $I\subset\{1,\dots,d\}$, the expression $\x_I$ denotes an $|I|$-tuple of real numbers representing the components $x_j$ for $j\in I$. The domain of $\x_I$ is (typically) the hypercube $[{\bf 0}_I,\1_I]$.  Suppose that $I,J\subset\{1,\dots,d\}$ and $\x,\z\in[{\bf 0},\1]$ with $I\cap J =\emptyset$.  Then we define the concatenation symbol $:$ such that the vector $\x_I:\z_J$ represents the point $\y\in[{\bf 0}_{I\cup J}, \1_{I\cup J}]$ with $y_j = x_j$ for $j\in I$, and $y_j = z_j$ for $j\in J$. The vector $\x_I:\z_J$ is well defined for $\x_I\in[{\bf 0}_I, \1_I]$ and $\z_J\in[{\bf 0}_J,\1_J]$ when $I\cap J=\emptyset$, even if $\x_{-I}$ or $\z_{-J}$ is left unspecified.  We also use the concatenation symbol to glue together more than two sets of components.  For instance $\x_I:\y_J:\z_K\in[{\bf 0},\1]$ is well defined for $\x_I\in[{\bf 0}_I,\1_I]$, $\y_J\in[{\bf 0}_J,\1_J]$ and $\z_K\in[{\bf 0}_K,\1_K]$ when $I,J,K$ are mutually disjoint sets whose union is $\{1,\dots,d\}$.  The main purpose of the concatenation symbol is to construct the argument to a function by taking components from multiple sources.

For a function $f:[{\bf 0},\1]\rightarrow\mathbb{R}$, a set $I\subset\{1,\dots,d\}$ and a constant vector $\c_{-I}\in[{\bf 0}_{-I},\1_{-I}]$ we can define a function $g$ as a lower-dimensional projection of $f$ on $[{\bf 0}_I,\1_I]$ via $g(\x_I)=f(\x_I:\c_{-I})$.  We write $f(\x_I;\c_{-I})$ to denote such a function with the argument $\x_I$ on the left of the semicolon and the parameter $\c_{-I}$ on the right.
\\

\section{A general result}

The main theorem in this section states that weak convergence of  a stochastic process $\int f\, {\rm d}\GG_n$, with $f\in \F$, in $\ell^\infty(\F)$ follows from the weak convergence of the stochastic process $\GG_n$ to a continuous Gaussian process $\GG$ in $\ell^\infty([0,1]^d)$,  for a large class of functions $\F$.
This  result is of interest in its own, and it is essentially the $d$-dimensional version of an integration by parts trick introduced in Radulovi\'c and Wegkamp (2015). The proof relies on Proposition \ref{IBP!} that gives a very general integration by parts formula for $\int f\, \d \GG_n$. 
The main idea is to change the integration over $\GG_n$ by integration over $f$.  For this reason we consider functions $f$ for which we can uniquely define a (signed) Borel measure on $[0,1]^d$.  The classical Lebesgue-Stieltjes integration theory on $\mathbb{R}$ is based on functions $f$ that are of bounded variation.  To consider its multivariate extension, naturally we will need to consider multivariate extensions of the concept of bounded variation.

First, we briefly recall the definition of total variation in the sense of Vitali, and refer to
Owen (2005) for a lucid presentation.
Following Owen (2005), a {\em ladder}  $\Y$ of $[0,1]$ is a (possibly empty) set of finitely many points in $(0,1)$. Each element $y\in\Y$ has a successor $y^+$, defined as the smallest element in $(y,1)\cap \Y$. (If the intersection is empty, we set $y^+=1$).
A multivariate ladder $\Y=\prod _{i=1}^d \Y_i$ of $[0,1]^d$  is based on $d$ one-dimensional ladders $\Y_i$ of $[{\bf 0}_{\{i\}},\1_{\{i\}}]$, and a successor $\y^+$ of $\y\in\Y$ is defined by taking each coordinate $y^+_j$ to be the successor of $y_j$.
The variation of a function $f$ over the multivariate ladder $\Y$ is
\begin{eqnarray*}
V_\Y(f) = \sum_{y\in\Y} | \Delta_f((\y,\y^+]) |.
\end{eqnarray*}
Here $\Delta_f((\y,\y^+])$ is the generalized volume of the hypercube $(\y,\y^+]$ based on the measure $\Delta_f$ to be introduced in (\ref{eq:f_fun_measure}).
Then the total variation of $f$ in the sense of Vitali is
\begin{eqnarray*}
V(f) := \sup_{\Y} V_\Y(f) = \sup_{\Y} \sum_{y\in\Y} | \Delta_f((\y,\y^+]) |.
\end{eqnarray*}
Here the supremum is taken over all multivariate ladders $\Y=\prod_{i=1}^d \Y_i$ of $[0,1]^d$.
It can be shown that
\[ V(f) \le  \int_{[{\bf 0},\1]} \left| \frac{\partial^d}{\partial u_1\cdots\partial u_d} f(\u) \right| \, \d\u
\] provided the mixed partial derivative of $f$ exists, see, for instance, Owen (2005, Proposition 13).\\

We will also need to consider total variation in the sense of   Krause (1903a; 1903b) and Hardy (1905).
Formally, the total variation of a function $f$ in the sense of Hardy-Krause is
\begin{align}
V_{\textup{HK}}(f) =\sum_{I\subset\{1,\ldots,d\}}  V\left( f(\cdot;{\bf 1}_{-I}) \right)  =
\sum_{ I \subset\{1,\ldots,d\} }  \int_{[{\bf 0}_I,\1_I]} \, |{\d}\Delta_f(\x_I; {\bf 1}_{-I} )|. \nonumber
\end{align}
Here $V\left( f(\cdot;{\bf 1}_{-I}) \right)$ is the Vitali variation of the function $f(\cdot; {\bf 1}_{-I})$ over $[{\bf 0}_I,\1_I]$.  (We recall from Section~\ref{sec:notations} that the function $f(\cdot; {\bf 1}_{-I}):[{\bf 0}_I,\1_I]\rightarrow\mathbb{R}$ is the lower-dimensional projection of $f$ on $[{\bf 0}_I,\1_I]$ obtained by setting $f(\x_I; {\bf 1}_{-I})=f(\x_I:{\bf 1}_{-I})$.  Also note that, in contrast to the literature, for convenience we are also including a term corresponding to $I=\emptyset$ in the sum, although this choice makes no material difference later on because the class of functions with bounded Hardy-Krause variation remains the same under our definition.)


\ \

We will mostly consider functions satisfying the following assumption:\\

\noindent
{\bf Assumption~F.}
$f:[0,1]^d\rightarrow\mathbb{R}$ is right-continuous (to be precise, following Aistleitner~\&~Dick (2014), we say a function is right-continuous if it is coordinatewise right-continuous in each coordinate, at every point) and is of bounded variation in the sense of Hardy-Krause, that is, $V_{\textup{HK}}(f)<\infty$.\\

By Aistleitner~\&~Dick (2014, Theorem~3), if a function $f$ satisfies assumption~F, then there exists a unique signed Borel measure $\Delta_f$ on $[0,1]^d$ for which
\begin{align}
\Delta_f([{\bf 0},\x]) = f(\x),\quad \x\in[0,1]^d.
\label{eq:f_fun_measure}
\end{align}
From now on, for notational brevity, we will use the same letter $f$ to denote the function $f$ and the measure $\Delta_f$ to which it gives rise.  From (\ref{eq:f_fun_measure}), it is easy to see that, for $\a,\b\in[0,1]^d$ and $\a\le\b$, the measure $f$ assigns weight
\begin{align}
f((\a,\b]) = \int_{(\a,\b]}\, \d f = \sum_{I\subset \{1,\ldots,d\}} (-1)^{|I|} f(\a_I: \b_{-I})
\label{def:measure}
\end{align}
to the hypercube $(\a,\b]=(a_1,b_1]\times\dots\times(a_d,b_d]$.  In fact, we can conclude from Aistleitner~\&~Dick (2014) a more general result that we will also use later: if a function $f$ satisfies assumption~F, and if we let $I\subset\{1,\dots,d\}$, $\a,\b\in[0,1]^d$ with $\a_I\le\b_I$, and $\c\in\{0,1\}^d$, then to the lower-dimensional projection $f(\cdot;\c_{-I})$ there corresponds a unique signed Borel measure $f(\cdot;\c_{-I})$ on $[{\bf 0}_I,\1_I]$ that assigns weight
\begin{align}
f((\a_I,\b_I];\c_{-I}) = \int_{(\a_I,\b_I]}\, \d f(\cdot;\c_{-I}) = \sum_{I'\subset I} (-1)^{|I'|} f(\a_{I'}:\b_{I-I'}:\c_{-I})
\label{def:measure_projection}
\end{align}
to the hypercube $(\a_I,\b_I]$. The validity of this claim is verified in Appendix B. We will identify a function and its lower-dimensional projections uniquely with the measures satisfying (\ref{def:measure}) and (\ref{def:measure_projection}).


\begin{theorem}
\label{thm:GENERAL}
Let $\GG_n$ be a stochastic process such that its sample paths satisfy assumption~F almost surely, that $\GG_n(\u)=0$ almost surely, if $u_j=0$ for some $j\in\{1,\ldots,d\}$
, and that $\GG_n$ converges weakly to a continuous Gaussian process $\GG$ in $\ell^\infty([0,1]^d)$.
Let $\F$ be a class of  functions $f$ satisfying assumption~F with $V_{\textup{HK}}(f)\le T<\infty$. 
Then, the empirical process
$\int f\, {\rm d}\GG_n$, indexed by $f\in\F$, converges weakly to a Gaussian process in $\ell^\infty(\F)$.
\end{theorem}
\begin{proof}
See Section 2.1.
\end{proof}

For instance, empirical processes based on alpha-mixing sequences are covered by this result. Such a result is new, as weak convergence of empirical processes for dependent variables indexed by functions are sparse in the literature and typically require rather restrictive beta-mixing conditions. Rio (2000) proved weak convergence of the process
$\sqrt{n}(\FF_n-F)(\x)$ in $\ell^\infty([0,1]^d)$ under alpha-mixing conditions only.

\begin{cor}
Let $\X_i$, $i\in\ZZ$, be a stationary sequence of random variables in $[0,1]^d$ with continuous distribution $F$ and with alpha-mixing coefficients
\[  \alpha_k:= \sup\left\{ |\PP(A\cap B)-\PP(A)\PP(B)|,\ A\in \sigma(\X_j,\ j\le i),\ B\in \sigma(\X_{k+j},\ j\ge i),\ i\in\ZZ\right\}
\] satisfying
\[ \alpha_k=O(k^{-a}) \text{  for some $a>1$ and $k\to\infty$}.
\]
Let $\FF_n$ be the empirical distribution function based on $\X_i$, $i=1,\ldots,n$ and let $\GG_n=\sqrt{n}(\FF_n-F)$ be the standard empirical process in $\ell^\infty([0,1]^d)$.
Let $\F$ be a class of  functions $f$ satisfying assumption~F with $V_{\textup{HK}}(f)\le T$.
Then $\left\{ \int f\, {\rm d}\GG_n,\ f\in\F\right\}$  converges weakly to a Gaussian process in $\ell^\infty(\F)$.\end{cor}
\begin{proof}
Theorem~7.3 in Rio (2000) establishes the weak convergence of the process $\GG_n$ in $\ell^\infty([0,1]^d)$. The corollary follows immediately from Theorem~\ref{thm:GENERAL}.
 \end{proof}

\bigskip

\subsection{Proof of Theorem~1}

The proof of Theorem~1 
relies on 
the following integration by parts formula.

\begin{proposition}\label{IBP!}
Let $\GG_n(\u)$, $\u\in[0,1]^d$, be a stochastic process such that its sample paths satisfy assumption~F, and that
$\GG_n(\u)=0$ if $u_j=0$ for some $j$. For any $f$ satisfying assumption~F, we have
\begin{eqnarray}\label{IBPformula}
\int_{({\bf 0},\1]}f(\x)\,\mathrm{d}\mathbb{G}_{n}(\x) &=& \sum_{I\subset \{1,\ldots
,d\}}(-1)^{|I|}\int_{({\bf 0}_I, \1_I]}\mathbb{G}_{n}(\x_{I}-;{\bf 1}_{-I})\,\mathrm{d}%
f(\x_I;\1_{-I}).
\end{eqnarray}
\end{proposition}
\begin{proof}
The result follows from the general
formula (\ref{IBP-main}) in Theorem~\ref{IBP-main} in the appendix. Notice that in Theorem~\ref{IBP-main} if $I_2\ne\emptyset$, then each term
$\GG_n(\x_{I_1}-; {\bf 0}_{I_2}: {\bf 1}_{I_3})$ in the integrand of (\ref{IBP-main}) equals zero under the assumption of Proposition~\ref{IBP!}.
\end{proof}

For any $f\in \mathcal{F}$, we define
\begin{eqnarray*}
\bar{\mathbb{G}}_{n}(f) &=&\int f\,d\mathbb{G}_{n} \\
\widetilde{\mathbb{G}}_{n}(f) &=&\Gamma (\mathbb{G}_{n},f)
\end{eqnarray*} based on the functional
\begin{eqnarray*}
\Gamma(\GG_n,f) &:=& \sum_{I\subset \{1,\ldots,d\}}(-1)^{|I|}\int_{({\bf 0}_I, \1_I]}\mathbb{G}_{n}(\x_{I};{\bf 1}_{-I})\,\mathrm{d}%
f(\x_{I};{\bf 1}_{-I}).
\end{eqnarray*}
\\

First, for each $f\in\F$, the functional $\Gamma(\cdot,f):\ell^\infty(\RR^d)\to\RR$ is linear and Lipschitz as
\begin{eqnarray*}
| \Gamma (X,f)-\Gamma(Y,f) | &\le & \sum_{I\subset \{1,\ldots
,d\}} \int_{({\bf 0}_I, \1_I]}   | \,\mathrm{d}f(\x_{I};{\bf 1}_{-I})| \cdot  \| X-Y\|_\infty\\
&\le & T \| X-Y\|_\infty.
\end{eqnarray*}

For any fixed $f\in \mathcal{F}$,   by the
continuous mapping theorem, see, e.g., Theorem~1.3.6 in Van der Vaart \& Wellner (1996), and the weak convergence
of $\mathbb{G}_{n}\to\GG$, we have
\[
\widetilde{\mathbb{G}}_{n}(f)=\Gamma (\mathbb{G}_{n},f){\rightarrow }%
\Gamma (\mathbb{G},f):=\widetilde{\mathbb{G}}(f)
\]%
as $n\rightarrow \infty $. This result is pointwise in $f$; i.e., it provides {\em fidi}-convergence of $\widetilde\GG_n$. Linearity of $\Gamma(\cdot,f)$ yields that the limit $\widetilde \GG(f)$ is normal.\\

Next, we define the map $\Gamma:\ell^\infty(\RR^d)\to\ell^\infty(\F)$ as $\Gamma(X)= \Gamma(X,f), f\in\F$. (For notational brevity we use the same letter $\Gamma$ to denote the functional introduced earlier and the map here, though there should be no confusion because they take different arguments.) Then
\begin{eqnarray*}
 \|\Gamma(X)-\Gamma(Y)\| &=& \sup_{f\in\G}   |\Gamma(X,f)-\Gamma(Y,f)| \\
 &\le& T \| X-Y\|_\infty
 \end{eqnarray*}
The continuous mapping theorem guarantees that the limit $\widetilde\GG:= \Gamma(\GG)$ of $\Gamma(\GG_n)$ is tight in $\ell^\infty(\F)$.

Finally, we have for the bounded Lipschitz distance
\begin{eqnarray*}
d_{BL}(\bar{\mathbb{G}}_{n},\widetilde{\mathbb{G}}) &= & \sup_{H} \left|\EE [ H( \bar{\mathbb{G}}_{n} ) ] - \EE[ H( \widetilde{\mathbb{G}} )] \right|
\end{eqnarray*}
with the supremum taken over $H:\ell^\infty(\F)\to\RR$ with $\sup_{x\in \ell^\infty(\F)} | H(x)  |   \le 1$ and $\| H(x)-H(y)\|\le \| x-y\|$ for all $x,y\in\ell^\infty(\G)$,
the following bound
\begin{eqnarray*}
d_{BL}(\bar{\mathbb{G}}_{n},\widetilde{\mathbb{G}}) &\leq &d_{BL}(\bar{\mathbb{G}}%
_{n},\widetilde{\mathbb{G}}_{n})+d_{BL}(\widetilde{\mathbb{G}}_{n},\widetilde{\mathbb{G}}%
) \\
&\leq &T\mathbb{E}[\sup_{\x}|\mathbb{G}_{n}(\x)-\mathbb{G}%
_{n}(\x^{-})|]+Td_{BL}(\mathbb{G}_{n},\mathbb{G}).
\end{eqnarray*}%
The first bound follows since, first by applying Proposition~\ref{IBP!} to the term $\bar{\mathbb{G}}_{n}(f)$ and then by assumption on the uniform boundedness of $V_{\textup{HK}}(f)$, we have
\begin{eqnarray*}
|\bar{\mathbb{G}}_{n}(f)-\widetilde{\mathbb{G}}_{n}(f)| &\leq &\sup_{\x}|\mathbb{G
}_{n}(\x)-\mathbb{G}_{n}(\x^{-})|\sum_{ I \subset \{1,\ldots
,d\}}\int_{({\bf 0}_I, \1_I]}\,|\mathrm{d}f({\bf x}_{I}:1_{-I})| \\
&\leq &T\sup_{\x}|\mathbb{G}_{n}(\x)-\mathbb{G}_{n}(\x^{-})|,
\end{eqnarray*}%
while the bound for the second term is a consequence of
Lipschitz property of the map $\Gamma$ with Lipschitz constant $T$:
\begin{eqnarray*}
d_{BL}(\widetilde\GG_n,\widetilde \GG) &=& \sup_{H} \left| \EE[ H(\widetilde \GG_n)]- \EE [ H(\widetilde \GG)] \right| \\
&=& \sup_{H} \left| \EE[ H \circ \Gamma( \GG_n)]- \EE [ H \circ \Gamma( \widetilde \GG)] \right| \\
&\le& T d_{BL}( \GG_n , \GG).
\end{eqnarray*}
We conclude that $d_{BL}(\bar{\mathbb{G}}_{n},\widetilde{\mathbb{G}})\rightarrow 0$
as $n\rightarrow \infty $. Since the limit $\widetilde{\mathbb{G}}:=\Gamma(\GG)$ is tight,
the desired weak convergence of $\bar\GG_n$ and $\widetilde{\mathbb{G}}_{n}$ follows.
\qed
 \bigskip

\section{Empirical copula processes indexed by functions}

\subsection{Smooth copula functions}
Our first result requires that the empirical process $\sqrt{n}(\bar\CC_n-C)(\u)$ converges weakly to a Gaussian limit  and
consider the class $\G$  of  right-continuous functions $g:[0,1]^d\to\RR$ with
\begin{eqnarray}\label{TV}
 V_{\textup{HK}}(g) :=  \sum_{ I \subset\{1,\ldots,d\}}   \int _{[{\bf 0}_I, \1_I]} \, |{\rm d}g(\x_I; {\bf 1}_{-I})| \le T
\end{eqnarray}
In words, we require that the Vitali variation of the functions $g$ and their marginals $g(\x_I;{\bf 1}_{-I})$ are uniformly bounded.

\begin{theorem}\label{thm:ecp1}
Assume that $\sqrt{n}(\bar\CC_n-C)$ converges weakly to a continuous Gaussian process in $[0,1]^d$.
Provided the functions $g\in\G$ satisfy assumption~F with  $V_{\textup{HK}}(g)\le T<\infty$, the empirical process $\bar\ZZ_n$, defined in (\ref{bar Z_n}), converges weakly to a Gaussian limit in $\ell^\infty(\G)$.
\end{theorem}
\begin{proof}
This follows immediately from Theorem~1.
\end{proof}

The proof of Theorem~1 reveals that limiting process can be characterized as  
\begin{eqnarray*}
\Gamma(\GG,f) &:=& \sum_{I\subset \{1,\ldots,d\}}(-1)^{|I|}\int_{({\bf 0}_I, \1_I]}\mathbb{G}(\x_{I};{\bf 1}_{-I})\,\mathrm{d}%
f(\x_{I};{\bf 1}_{-I})
\end{eqnarray*}
for $f\in\G$, based on
 the limit $\GG$ of $\sqrt{n}(\bar{\CC}_n-C)$.  \\

The class of functions $\G$ considered in Theorem~\ref{thm:ecp1} is an obvious generalization of the class of indicator functions $1\{ \cdot\le \x\}$ of the half spaces $\prod_{i=1}^d (0,x_i]$, $\x=(x_,\ldots,x_d)\in[0,1]^d$. Theorem~\ref{thm:ecp1} requires no differentiability of $g\in\G$, only right-continuity and bounded variation.\\

Theorem~\ref{thm:ecp1}  requires that the empirical process $\sqrt{n}(\CC_n-C)$ converges weakly in $[0,1]^d$. This is shown in increasing generality by Stute (1984), Fermanian et al. (2004) and Segers (2012). It also allows for dependent observations $\X_i$, not just i.i.d. observations, as the theorem requires weak convergence only of the process $\sqrt{n}(\CC_n-C)$.
B\"ucher and Vogulshev (2013), in turn, show that the latter  is implied by weak convergence of the process
$\sqrt{n}(\FF_n-C)$ for $\FF_n$ the empirical distribution function based on pseudo-observations ${\bf U}_1,\ldots,{\bf U}_n$ with
${ \bf U}_{i}= (F_1(X_{i1}),\ldots,F_d(X_{id})$.

 \begin{cor}\label{aaaa}
Assume that $\sqrt{n}(\FF_n-C)$ converges weakly to a Gaussian limit $B_C$ in $\ell^\infty([0,1]^d)$ with
$B_C$ continuous and $B_C({\bf 1})=0$ and $B_C(\x)=0$ if $x_j=0$ for some $j$. Moreover, assume that $\dot{C}_k$ exists and is continuous on $\{\u\in[0,1]^d:\ u_k\in(0,1)\}$ for $k=1,\ldots,d$.
Provided the functions $g\in\G$ satisfy assumption~F with  $V_{\textup{HK}}(g)\le T<\infty$, the empirical process $\bar\ZZ_n$, defined in (\ref{bar Z_n}), converges weakly to a Gaussian limit in $\ell^\infty(\G)$.
\end{cor}
\begin{proof}
Corollary 2.5 in B\"ucher and Vogulshev (2013) prove
weak convergence of  $\sqrt{n}(\FF_n-C)$ in $\ell^\infty([0,1]^d) $  and the conditions on its limit imply
weak convergence of the empirical copula process
$\sqrt{n}(\CC_n-C)$ to a continuous Gaussian process in $\ell^\infty([0,1]^d)$.
The conclusion follows immediately from Theorem~\ref{thm:ecp1}.
\end{proof}

For instance, if the empirical distribution $\FF_n$ is based on  a stationary  sequence of  $\X_i$  satisfying the (alpha-mixing) conditions in Corollary 2, the conclusion of  Corollary \ref{aaaa} holds.
\\
The  limiting process of $\sqrt{n}\int g\, d(\bar \CC_n-C)$ can be characterized in the i.i.d. case as  
\begin{eqnarray*}
\sum_{I\subset \{1,\ldots,d\}}(-1)^{|I|}\int_{({\bf 0}_I, \1_I]} \left\{ \alpha(\u_I;{\bf 1}_{-I}) -\sum_{i\in I}  \dot{C}_i(\u_I;{\bf 1}_{-I}) \alpha_i(\u_I;{\bf 1}_{-I}) \right\}   \,\mathrm{d}%
f(\u_{I};{\bf 1}_{-I})
\end{eqnarray*}
for $f\in\G$.   
Here  $\alpha -\sum_{j=1}^d \dot{C}_j \alpha_j$  is  the limiting process of $\sqrt{n}(\bar{\CC}_n-C)$ and $\alpha$ is a $C$-Brownian bridge in $\ell^\infty([0,1]^d)$.

\bigskip

\subsection{Smooth index functions}
Our next result requires that $\G$ is a $C$-Donsker class of differentiable functions $g:[0,1]^d\to\RR$.
For any $g\in\G$,  we write  $\dot{g}_k$ be the partial derivative of $g$ with respect to the $k$th coordinate, i.e., $\dot{g}_k(\u)=\partial_k g(\u)= \partial g(\u)/\partial u_k$, $\u=(u_1,\ldots,u_d)$.
We assume that the classes of partial derivatives
\begin{eqnarray}\label{eq:calF_k}
 \dot{\calG}_k = \left\{
\dot{g}_k = \partial_k g,\ g\in\calG \right\}
\end{eqnarray}
are uniformly equicontinuous.
Interestingly, if the functions $g$ are sufficiently  smooth, then existence of first-order partial derivatives of $C$ is no longer required for the weak convergence of $\bar \ZZ_n$.\\

\begin{theorem}\label{theorem:ecp2}
Assume that
\begin{itemize}
\item[-] $F$ has continuous marginals, and copula function $C$;
\item[-]  $\G$ is a uniformly bounded $C$-Donsker class;
\item[-]
  the first-order partial derivatives $\dot{g}_k$ of  $g\in\G$ exist and the   classes $\dot{\G}_k$, $k=1,\ldots,d$, are uniformly equicontinuous and   uniformly bounded.
 \end{itemize}
Then, the empirical copula process $\bar\ZZ_n$, defined in (\ref{bar Z_n}),
converges weakly to a Gaussian process in $\ell^\infty(\G)$, as $n\to\infty$.
\end{theorem}
\begin{proof}
See section 4.2.
\end{proof}

\bigskip

\subsection*{Discussion of the conditions of Theorem~\ref{theorem:ecp2}}
\begin{itemize}

\item It follows from the proof that the limiting process can be characterized as
\begin{eqnarray*}
\bar\ZZ_n &=& \int g(\u)\, d\alpha(\u) + \sum_{k=1}^d \int \dot{g}_k(\u)\alpha_k(u_k)\, dC(\u)
\end{eqnarray*}
for the limiting $C$-Brownian bridge $\alpha$ in $\ell^\infty([0,1]^d)$ of the empirical   process $\sqrt{n}(\FF_n-C)$  based on pseudo-observations ${\bf U}_1,\ldots,{\bf U}_n$ with
${ \bf U}_{i}= (F_1(X_{i1}),\ldots,F_d(X_{id}))$.\\

\item Theorem~\ref{theorem:ecp2} is slightly more general than Corollary 5.4 in Van der Vaart \& Wellner (2007). It corrects a slight mistake in their proof. While they require that the partial derivatives $\dot{g}_k$ are continuous, their proof requires that they are in fact uniformly equicontinuous. In fact, at page 247, line 13 they require convergence, uniformly in $g$, while their proof of this fact (Lemma 4.1 at the same page) only gives pointwise convergence. While this  is easily fixed,  the other difference with their result, however,  is that we do not require that the uniform entropy integral  $J(1,\G,L_2)$ is finite, which requires an altogether  different proof.
\\

\item
The alert reader may wonder if the uniformly bounded
assumption on the classes $\G$ and $\dot{\G}_k$ may be replaced by suitable envelope conditions. However, if the class $\dot{\G}_k$ of uniformly equicontinuous functions    $f:[0,1]^d\to\RR$ has an integrable envelope, then it must be uniformly bounded on $[0,1]^d$.
A similar reasoning  holds for $\G$:
since the domain of the functions is $[0,1]^d$,   the assumption that $\G$ has an integrable envelope, coupled with the fact that the partial derivatives exist  and are uniformly bounded, immediately forces that all $g\in\G$ must be uniformly bounded. \\

\item
 It is remarkable that Theorem~\ref{theorem:ecp2} holds without any condition on $C$, under rather mild regularity on the functions $g$. This is in contrast
 with the required smoothness assumptions on $C$ for the ordinary empirical copula process (indexed by boxes) in (2).\\
Arguably the best known examples of non-differentiable copulas are
the Marshal-Olkin copula $C(u,v) = \min(u^{1-\alpha} v, u v^{1-\beta})$,
and the Frechet-Hoeffding copulas $C(u,v)=\max(u+v-1,0)$ and $C(u,v)=\min(u,v)$.
Another example is the Cuadras-Aug\'e copula given by
\[ C(u,v)= \{ \min(u,v)\}^\theta \{uv\}^{1-\theta},\quad 0\le \theta\le1.\]
A common   technique
 to construct a copula from a given function $\delta:[0,1]\to[0,1]$ yields non-differentiable
copulas as well  by setting
 \[C(u,v) = \min\left[u, v, \{\delta(u)+\delta(v) \}/2 \right]\]
  or
    \begin{align*}
C(u,v)& = \left\{ \begin{array}{c l}
\displaystyle
u-\inf_{u\le x\le v}  \{ x-\delta(x)\}&\text{if} \  u\le v \\
\displaystyle v-\inf_{v\le x\le u}  \{ x-\delta(x)\}&\text{if} \  u > v .
\end{array}  \right.\\
\end{align*}

 \item
A natural class of functions to consider is $C_1^s([0,1]^d)$, as described in detail by Van der Vaart \& Wellner (1996), pp 154--157.
These are all functions on $[0,1]^d$ that have uniformly bounded partial derivatives up to order $\lfloor s \rfloor$ and the highest partial derivatives are H\"older of order $s-\lfloor s \rfloor$.  Theorem~2.7.1 and Theorem~2.7.2 in Van der Vaart \& Wellner (1996) show
this class $C_1^s([0,1]^d)$ is universally Donsker if $s>d/2$. In particular, this means that for $d=2$, the processes $\ZZ_n$ and $\bar{\ZZ}_n$ converge weakly in $\ell^\infty( C_1^s([0,1]^2)) $, provided the smoothness index $s>1$, that is, all functions have partial derivatives that satisfy a uniform H\"older condition of any order.\\
 \end{itemize}

 \subsection{Bootstrap empirical copula processes}
We provide the   bootstrap counterpart of  Theorems  \ref{thm:ecp1} \& \ref{theorem:ecp2}.
Let the  bootstrap sample $(\X_1^*,\ldots,\X_n^*)$ be obtained by sampling with replacement from
$\X_1$,$\ldots$, $\X_n$. We write
\begin{eqnarray}
\FF_n^*(\x)=\frac1n \sum_{i=1}^n \1\{ \X_i^*\le \x \},\ \x\in\RR^d,
\end{eqnarray}
 for the empirical cdf based on the bootstrap, with marginals
 \begin{eqnarray}
 \FF_{nj}^*(t)= \frac1n \sum_{i=1}^n \1 \{ X_{ij}^*\le t\}, \ t\in\RR,\
  j=1,\ldots,d.
 \end{eqnarray}
We denote its associated empirical copula function by $\CC_n^*$ and
\begin{eqnarray}
\bar \CC_n^*(\u) &=& \frac{1}{n} \sum_{i=1}^n 1\{ \FF_{n1}^*(X^*_ {i1})\le u_1,\ldots,\FF^*_{nd}(X^*_{id})\le u_d \}, \qquad \u\in [0,1]^d.
\end{eqnarray}
For the bootstrap
  empirical copula process
\begin{eqnarray}
\label{Z_n^*}
\bar{ \ZZ}_n^*(g)= \sqrt{n}\int g(\u)\, {\rm d} (\bar{\CC}_n ^*-\bar{\CC}_n) (\u), \ g\in\G \end{eqnarray}
we have the following
bootstrap version of Theorems \ref{thm:ecp1}  \& \ref{theorem:ecp2}.

 \begin{theorem} \label{boot}
Under the conditions of either  Theorem~\ref{thm:ecp1}  or Theorem~\ref{theorem:ecp2}, the conditional distribution of $\{ \bar\ZZ_n^*(g),\ g\in\G\}$ converges weakly to the same Gaussian limit as $\{ \bar\ZZ_n(g),\ g\in \G\}$, in probability.
\end{theorem}

  More precisely, we prove that
\begin{eqnarray}\label{w}
\lim_{n\to\infty} \EE \left[  \sup_{h} \left| \EE [ h(\bar\ZZ_n) ] - \EE^* [ h(\bar\ZZ_n^*) ] \right| \right]=0.
\end{eqnarray}
Here $\EE^*$ is the conditional  expectation with respect to the bootstrap sample and
the supremum in (\ref{w}) is taken over 
 all uniformly bounded, Lipschitz   functionals $h:\ell^\infty(\F_n)\to\RR$ with Lipschitz constant 1, that is,
\begin{eqnarray}\label{BL1}
\sup_{x\in \ell^\infty(\G)} |h(x)|\le 1\end{eqnarray}
 and, for all $x,y\in \ell^\infty(\G)$,
\begin{eqnarray}
|h(x)-h(y)|\le   \sup_{g\in\G} | x(g)- y(g) |.\label{BL2}
\end{eqnarray}
{As usual in the empirical process literature, it is tacitly understood that we take outer probability measures  whenever measurability issues arise.}

\begin{proof}
See Section 4.4.
\end{proof}

\bigskip
The bootstrap approximation can be used to obtain asymptotic uniform confidence bands for the copula function. \\
The proof of Theorem 7 shows that under the conditions of  Theorem~\ref{thm:ecp1}, we do not need the iid assumption, only weak convergence of both $\sqrt{n}(\bar{\CC}_n-C)$ and $\sqrt{n}(\bar{\CC}_n^*-\bar{\CC}_n)$ to the same Gaussian limit  is required.\\

\subsection{Some applications}
\begin{itemize}
\item[]{\bf Semi-parametric MLE.}
This type of results is useful in the same way the extension of the empirical process indexed by general Donsker classes from
indicator functions on the half-spaces $(-\infty, x]$, $x\in\RR^d$, has proved extremely useful. See, for instance, the monograph  Van der Vaart \& Wellner~(1996).
In the context of copula estimation,
an important  example is  the following semi-parametric maximum likelihood estimation problem (Tsukahara~2005). Suppose that  the copula  $C$ is parametrized by a finite dimensional parameter $\theta\in\Theta$, a subset of $ \RR^k$, with density $c_\theta$ and that the marginal distributions $F_j$ have densities $f_j$. The log-likelihood function in this setting is
\[ \log \ell (\theta) = \sum_{i=1}^n \log c_\theta( F_1( X_{i1}),\ldots, F_d(X_{id})) + \sum_{i=1}^n \sum_{j=1}^d \log f_j( X_{ij})
\]
and a common strategy therefore is to replace the unknown marginals $F_j$ by $\FF_{nj}$ and maximize
\[ \sum_{i=1}^n \log c_\theta( \FF_{n1}( X_{i1}),\ldots, \FF_{nd}(X_{id}))
\] over $\theta$. Assuming we can take the derivative with respect to $\theta$, we define
\[ \Psi(\theta)=\int \phi_\theta(\u)\, {\rm d}C(\u)\]
and
\[ \Psi_n(\theta) =\int \phi_\theta(\u)\, {\rm d}\bar\CC_n(\u).\]
We emphasize  that $\Psi_n$ is an integral with respect to $\bar{\CC}_n$, not $\CC_n$.  Here $\phi_{\theta}$ is the derivative of $\log c_\theta$ with respect to $\theta$.

Van der Vaart~\&~Wellner~(1996, Example 3.9.35) show that the solution  $\wh\theta_n$ of $\Psi_n(\theta)=0$ is asymptotically normal, provided the process $\sqrt{n}(\Psi_n-\Psi)(\theta)$ converges in distribution to a Gaussian $\ZZ$ with continuous sample paths  in $\ell^\infty(\Theta)$ and regularity of $\Psi$ ($\Psi(\theta)=0$ has a unique solution $\theta_0$, $\Psi$ is a local homeomorphism at $\theta_0$,  differentiable  at $\theta_0$ with derivative $\dot{\Psi}_{\theta_0}$).
Consequently, if    the class of functions $\phi_\theta$ indexed by $\theta\in\Theta$ satisfies either  the conditions of Theorem~\ref{thm:ecp1}
and $\sqrt{n}(\CC_n-C)$ converges weakly, or the conditions of Theorem~\ref{theorem:ecp2}
  (with no assumptions on $C$),
and
\[ \lim_{\|\theta' - \theta\|\to0} \int ( \phi_\theta- \phi_{\theta'} )^2 \, {\rm d} C= 0,\]
and $\Psi$ satisfies the regularity conditions above, then $\wh\theta$ is asymptotically normal.
\\

\item[]{\bf Testing of non-smooth copulas.}
The usual Kolmogorov-Smirnov test statistic
\[\sqrt{n} \sup_{\u} | \CC_n(\u)- C(\u) |
\]
converges provided $C$ is sufficiently regular (conform Segers (2012) conditions).
 If we want to test for a non-smooth $C$,  one that   does not meet the mild condition 4.3 of  B\"ucher, Segers and Vogulshev  (2014), then  Theorem~\ref{theorem:ecp2}  poses a solution by considering
 \[ \sqrt{n} \sup_{g\in \G} \left|\int g\,{\rm d}( \CC_n- C) \right|
\]
 for a sufficiently rich class $\G$ instead. For instance, the class of all differentiable functions $g$ with   Lipschitz  partial derivatives on $[0,1]^d$ is (universally) Donsker, whilst it is rich enough for our testing purposes as it characterizes weak convergence.\\
 From a computational point of view, we may consider the class $g(\x)=g_\t(\x)=\exp(<\t,\x>)$, with $\t\in[0,1]^d$ so that we compare the moment generating functions (which are defined for any copula, as the random variables are bounded).  Indeed, if the function $\int e^{<\t,\u>}\, {\rm d} C(\u)$ is piecewise differentiable in $\t$, then this would lead to  an easily computable test statistic and  a consistent test. 
   \\

\end{itemize}

\bigskip

\section{Proofs of Theorems~\ref{theorem:ecp2} \&~\ref{boot} }

\subsection{Notation}

Throughout, we assume
without loss of generality that all marginals $F_j$ are uniform distributions,  $j=1,\ldots,d$. This implies that $F=C$.
This common simplification in the copula literature  is justified by, for instance,   Lemma 8 of  Fermanian et al (2013).
Indeed, $\ZZ_n(g)$ and $\bar{\ZZ}_n(g)$  remain the same if we replace the original observations $\X_i=(X_{i1},\ldots,X_{id})$ by
 the pseudo-observations ${\bf Y}_i=(F_1(X_{i1}),\ldots,F_d(X_{id}))$, $i=1,\ldots,n$. Observe that, indeed, the distribution function of each ${\bf Y}_i$ is the copula $C$ and each marginal $Y_{ij}$ is uniformly distributed on $[0,1]$, $i=1,\ldots,n$, $j=1,\ldots,d$.
Having made this blanket assumption ($F_j(x)=x$, $j=1,\ldots,d$), we denote by  ${\UU}_n$   the empirical process $\sqrt{n}(\FF_n-F)$ in $\ell^\infty([0,1]^d)$ with marginals ${\UU}_{nj}=\sqrt{n}(\FF_{nj}-F_j)$, $j\in\{1,\ldots,d\}$. \\

\subsection{Proof of Theorem~\ref{theorem:ecp2}}

For any $g\in\G$,  we write  $\dot{g}_k$ be the partial derivative of $g$ with respect to the $k$th coordinate and we define, for $k\in\{1,\ldots,d\}$,
 the classes
\begin{align}
\label{eq:calF_int_k}
\calG_{\textit{int},k} &= \left\{ T_{k}(g) : g\in\calG \right\}
\intertext{based on  the functions}
T_{k}(g)(\x) &= \int \dot{g}_k(\u) \1\{x_k\le u_k\} \, {\rm d}C(\u).
\end{align}
We define the empirical process
\begin{align}
\widetilde{\ZZ}_n(g) &= \int  \left[ g + \sum_{k=1}^d T_{k}(g)  \right]\, {\rm d} {\UU}_n.
\end{align}
Lemma \ref{convergentie} shows that $\widetilde{\ZZ}_n$ converges weakly, and it suffices to show that $\bar\ZZ_n$ and $\widetilde{\ZZ}_n$ are asymptotically equivalent, as $n\to\infty$.
 Some simple algebra shows that
\begin{align}
\int  T_{k}(g)\, {\rm d} \FF_n
&= \dfrac{1}{n} \sum_{i=1}^n \int \dot{g}_k(\x) \1\{X_{ik}\le x_k\}\, {\rm  d}C(\x) \nonumber \\
&= \int \dot{g}_k(\x) \FF_{nk}(x_k) \, {\rm d}C(\x), \nonumber \\
\intertext{and}
\int  T_{k}(g)\, {\rm  d}C  &= \EE\left[  \int \dot{g}_k(\x) \1\{X_k\le x_k\} \, {\rm d}C(\x)\right] \nonumber \\
&= \int \dot{g}_k(\x) F_k(x_k)\,  {\rm d}C(\x), \nonumber
\end{align}
so  that
\begin{align}
 \int T_{k}(g) \, {\rm d} {\UU}_n  &= \int \dot{g}_k(\x) {\UU}_{nk}(x_k) {\rm d}C(\x). \nonumber
\end{align}
It is now easily verified that
\begin{align}
(\bar{\ZZ}_n - \widetilde{\ZZ}_n)(g) = I(g) + \textit{II}(g)  \nonumber
\end{align}
for
\begin{align}
\textit{I}(g) &= \int \left[ \sqrt{n} \left[ g\left( \FF_{n1}(x_1),\dots,\FF_{nd}(x_d) \right) - g(\x) \right] - \sum_{k=1}^d \dot{g}_k(\x){\UU}_{nk}(x_k) \right] \, {\rm d}\FF_n(\x) \nonumber \\
\textit{II} (g)&=  \int \left[ \sum_{k=1}^d \dot{g}_k(\x) {\UU}_{nk} (x_k) \right] \, {\rm d} n^{-1/2} {\UU}_n(\x) \nonumber .
\end{align}
Hence, if
 $$\sup_{g\in\G}  | \textit{I} (g)+\textit{ II} (g) \, |\rightarrow 0,$$ in probability, as $n\to\infty$, then $\bar{\ZZ}_n$ converges weakly to the same limit as $\widetilde{\ZZ}_n$.  This is verified in the Propositions \ref{prop1} \& \ref{prop2}, and the proof of Theorem~1 is complete.
 \qed

\bigskip

\begin{lemma}\label{convergentie}
Under the assumptions of Theorem~\ref{theorem:ecp2}, the empirical process $\widetilde{\ZZ}_n$ converges weakly. 
\end{lemma}
\begin{proof}
The class $$\calG'=\left\{g + \sum_{k=1}^d T_k(g):\ g\in\calG\right\}$$ is a subset of the class $$\calG''=\left\{g + \sum_{k=1}^d t_k:g\in\calG,\ t_k\in\calG_{\textit{int},k}\right\}.$$
By definition,  the class $\calG$ is $C$-Donsker and the classes $\dot{\calG}_k$, $k\in\{1,\dots,d\}$, are  uniformly equicontinuous. This implies that the classes $\calG_{\textit{int},k}$, $k\in\{1,\dots,d\}$, are $C$-Donsker.
This in turn implies that
 the class $\calG''$ is   $C$-Donsker by Theorem~2.10.6 of van~der~Vaart~\&~Wellner (1996), as the pointwise sum of two Donsker classes is again Donsker.
\end{proof}

\bigskip

 \begin{proposition}
\label{prop1}
Under the assumptions of Theorem~\ref{theorem:ecp2},
we have
\begin{align}
\sup_{g\in \G} | \textit{I}(g) \, |\inprobto0, \text{ as } n\to\infty.
\label{eq:II_convergence_in_prob}
\end{align}
\end{proposition}

\begin{proof}
We have
\begin{align}
\sup_{g\in\G} \left| \textit{I}(g) \right| \le \sum_{k=1}^d \sup_{g\in\G} \left| \textit{I}_k(g) \right| \nonumber
\end{align}
for
\begin{align}
\textit{I}_k(g) = \int \left[ \left( \dot{g}_k(\wt \X_{n,\x}) - \dot{g}_k(\x) \right) {\UU}_{nk}(x_k) \right] \, {\rm d} \FF_n(\x). \nonumber
\end{align}
Here   $\wt\X_{n,\x}$ are  (random) points  on the line segment between $\x$ and $(\FF_{n1}(x_1),\cdots,\FF_{nd}(x_d))^T$, and we
used the mean value theorem.
Hence, it suffices to prove that
\begin{align}
\sup_{g\in\G}\left| \textit{I}_k(g) \right| \overset{\mbox{\tiny{P}}}{\longrightarrow} 0, \text{ as } n\to\infty
\label{eq:II_k_convergence_in_prob}
\end{align}
for each $k\in\{1,\dots,d\}$.
By Lemma \ref{lemma:uniform_equicontinuity_function_phi}  below, there exists a bounded, non-negative, and monotone increasing function $\phi_k(t)$ with $ \lim_{t\downarrow 0} \phi_k(t)=0$ such that
\begin{eqnarray*}
 \sup_{g\in\G} \left| \dot{g}_k(\wt \X_{n,\x}) - \dot{g}_k(\x) \right| &\le& \phi_k(\| \wt \X_{n,\x} - \x\| )\\
 &\le& \phi_k(\|  \FF_n(\x) - \x\| ) \\
& \le& \phi_k(\|  n^{-1/2}\UU_n\|_\infty ) ,
 \end{eqnarray*}
whence
\begin{eqnarray*}
\sup_{g\in\G}\left| \textit{I}_k(g) \right|
 &\le&  \| {\UU}_{nk} \|_\infty {\phi_k(\|  n^{-1/2}\UU_n\|_\infty )} \int  \, {\rm d} \FF_n(\x)\\
 & =&  \| {\UU}_{nk} \|_\infty  \phi_k( \| n^{-1/2} \UU_n\|_\infty).
\end{eqnarray*}
The empirical process $\UU_{nk}$ converges weakly and hence $ \| {\UU}_{nk} \|_\infty =O_p(1)$. By the Glivenko-Cantelli theorem in $\RR^d$, $\| n^{-1/2} \UU_n\|_\infty=o_p(1)$, and hence $\phi_k( \| n^{-1/2} \UU_n\|_\infty)=o_p(1)$. We conclude that (\ref{eq:II_k_convergence_in_prob}) holds for every $k\in\{1,\ldots,d\}$ and hence (\ref{eq:II_convergence_in_prob}) is verified.
\end{proof}

\bigskip

\begin{proposition}\label{prop2}
Under the assumptions of Theorem~\ref{theorem:ecp2},
 we have
\begin{align}
\sup_{g\in \G} | \textit{II}(g) \, |\inprobto0, \text{ as } n\to\infty.
\label{eq:I_b_convergence_in_prob}
\end{align}
\end{proposition}

\begin{proof}
It suffices to show that
\begin{align}
\sup_{g\in\G} \left| \textit{II}_{k}(g)  \right| \overset{\mbox{\tiny{P}}}{\longrightarrow} 0,  \text{ as } n\to\infty \nonumber
\end{align}
for each $k\in\{1,\dots,d\}$, for
\begin{align}
\textit{II}_{k}(g) = \int \left[ \dot{g}_k(\x) \UU_{nk}(x_k) \right] dn^{-1/2} \UU_n(\x) . \nonumber
\end{align}
We define the class of functions
\begin{align}
\calD_n(M) &= \left\{ D: D~\text{is a c.d.f. on $[0,1]$ with } \sqrt{n} \|D-I\|_{\infty} \le M  \right\},  \label{D_n} \\
\calH_{k,n}(M) &= \left\{h= \sqrt{n} (D-I) f_k:\  f_k\in\dot{\calG}_k,\ D\in\calD_n(M) \right\}. \label{H_n}
\end{align}
Fix an arbitrary (small) $\eps\in(0,1)$. There exists $M=M(\eps)<\infty$ such that
\[ \limsup_{n\to\infty} \PP \{ \|\UU_{nk}\|_{\infty}\ge  M\} \le \eps.
\]
On the event $\{ \| \UU_{nk}\|_\infty \le M\}$, we have
\begin{align}\label{eq:ULLN_calG_k_n}
\sup_{g\in \G} | \textit{II}_{k}(g) |\le \sup_{h\in \calH_{k,n}(M)} \left | \int h\,{ \rm d} n^{-1/2} \UU_n  \right|
\end{align}
and
to prove the proposition, it suffices to verify  that the term on the right converges to zero, in probability, $n\to\infty$.
By a straightforward modification of Theorem~2.4.3 of Van der Vaart \& Wellner (1996), the right-hand side of
 (\ref{eq:ULLN_calG_k_n}) converges to zero,  if
 \begin{enumerate}
 \item the class
 $\calH_{k,n}(M)$ has an integrable envelope and
 \item
 for all $\xi>0$,
\begin{align}
\log N(\xi,\calH_{k,n}(M),L_1(\FF_n)) = o_p(n) \nonumber
\end{align}
holds. Here $N(\xi,\calH_{k,n}(M),L_1(\FF_n))$ is the $\xi$-covering number of $\calH_{k,n}(M)$ in $L_1(\FF_n)$, that is, the number of closed balls of radius $\xi$ in $L_1(\FF_n)$ needed to cover $\calH_{k,n}(M)$.
\end{enumerate}
Since $\dot{\calG}_k$ is uniformly bounded,  $\sup_{f_k\in\dot{\calG}_k}\| f_k\|_\infty \le M_k$ for some $M_k<\infty$, and we find
\[ \sup_{h\in \calH_{k,n}(M)} \| h\|_\infty \le M{ \cdot M_k},\]
so the envelope condition is fulfilled. We now verify that the metric entropy condition holds.
We fix arbitrary $h,h'\in\calH_{k,n}(M)$, and write
\begin{align}
h  &= \sqrt{n} (D -I) f_k, \nonumber \\
h' &= \sqrt{n} (D'-I) f'_k \nonumber
\end{align}
for $f_k, f'_k\in\dot{\calG}_k$ and $D,D'\in\calD_n(M)$.  We can easily deduce that, for any probability measure $Q$,
\begin{align}
\int | h - h' |\, {\rm d}Q  &\le \sqrt{n} M_k  \int |D-D'|\, {\rm d}Q     +  M \int| f_k-f'_k |\, {\rm d}Q
. \nonumber
\end{align}
Hence, we conclude that, for any probability measure $Q$ and $\xi>0$,
\begin{align}
\log N(\xi,\calH_{k,n}(M),L_1(Q)) &\le \log N(\xi/(2 M_k \sqrt{n}),\calD_n,L_1(Q) )+ \log N(\xi/(2M),{  \dot \calG_k},L_1(Q))
 \nonumber \\
&\le  \log N(\xi/(2 M_k \sqrt{n}),\calD_n(M),L_1(Q) )+ \log N_\infty(\xi/(2M),\dot{\calG}_k).
\label{eq:entropy_intermediate_1}
\end{align}
Here $N_\infty(\eps,\dot{\G}_k)$ is the $\eps$-covering number of $\dot{\G}_k$ in $L_\infty([0,1]^d)$.
By Lemma~\ref{lemma:entropy_uniformly_equicontinuity_calF_k} and Lemma~\ref{mono} in the appendix, we have, from~(\ref{eq:entropy_intermediate_1}),
that
\begin{eqnarray*}
\log N(\xi,\calH_{k,n}(M),L_1(\FF_n)) &\le & \sup_Q \log N(\xi,\calH_{k,n}(M), L_1(Q)) \\
&\le&
K_1\sqrt{n} + K_2= {O}(\sqrt{n}) = o(n) \nonumber
\end{eqnarray*}with the supremum taken over all probability  measures $Q$, for some finite constants $K_1, K_2=K_2(\xi)$, independent of $n$.
This completes the proof.
\end{proof}

\bigskip

\subsection{Proof of Theorem~\ref{boot}}
{
Let $\UU_n^*=\sqrt{n}(\FF_n^*-\FF_n)$ be the bootstrap counterpart of $\UU_n=\sqrt{n}(\FF_n-F)$ with marginals $\UU_{nj}^*$, $j\in\{1,\dots,d\}$, and recall that $F=C$ as the marginal distributions  $F_j$ are uniform distributions on $[0,1]$.
The proof of the bootstrap counterpart of Theorem~\ref{thm:ecp1} is similar to the proof of Theorem~\ref{thm:ecp1}, after replacing the process $\GG_n$ in the proof of Theorem~\ref{thm:GENERAL} by $\UU_n^*=\sqrt{n}(\FF_n^*-\FF_n)$. For this reason, we concentrate on the proof of the bootstrap counterpart of Theorem~\ref{theorem:ecp2}.\\

We define the empirical processes
\begin{align}
\bar \ZZ_n^*(g) =& \sqrt{n} \int g(\x)\, d (\bar \CC_n^* - \bar \CC_n)(\x)\\
=& \sqrt{n} \left( \int g( \FF_{n1}^*(x_1),\ldots,\FF_{nd}^*(x_d) )\, {\rm d} \FF_n^*(\x) - \int g(\FF_{n1}(x_1),\ldots,\FF_{nd}(x_d) ) \, {\rm d} \FF_n (\x) \right)\nonumber
\\\intertext{and}
\widetilde \ZZ_n^*(g)
=&\int \left[ g(\x)+\sum_{k=1}^d T_{k}(g)(\x) \right]\, {\rm d}{\UU}_n^*(\x)\\
=& \int g(\x)\, {\rm d}{\UU}_n^*(\x) + \sum_k \int \dot{g}_k(\x) {\UU}_{nk}^*(x_k)\, {\rm d}C(\x)\nonumber
\end{align} with $g\in\G$.
The process $\widetilde \ZZ_n^*$  has a tight Gaussian limit, by the bootstrap CLT (see, for instance, Van der Vaart \& Wellner 1996, Theorem~3.6.1)
and Lemma \ref{convergentie}. Hence it suffices to show
\begin{eqnarray}
\label{eq}
\sup_{g\in\G} | \bar\ZZ_n^*(g) -\widetilde \ZZ_n^*(g) | =o_{p^*}(1), \text{ as } n\to\infty.
\end{eqnarray}
For this, we first observe that, after rearranging terms,
\begin{eqnarray*}
&& \int g(\FF_{n1}^*(x_1),\ldots,\FF^{*}_{nd}(x_d) )\, d\FF_n^*(\x) - \int g(\FF_{n1}(x_1),\ldots,\FF_{nd}(x_d))\, {\rm d}\FF_n(\x)
\\
&&=  \int g(\FF_{n1}^*(x_1),\ldots,\FF_{nd}^*(x_d))\, {\rm d}(\FF_n^*-\FF_n) (\x)\\
&&\quad +  \int g(\FF_{n1}^*(x_1),\ldots,\FF_{nd}^*(x_d)) -  g(\FF_{n1}(x_1),\ldots,\FF_{nd}(x_d))\,{\rm d}\FF_n(\x) \\
&&=  \int  g(\x) \, d(\FF_n^*-\FF_n)(\x) +   \int  \{ g(\FF_{n1}^*(x_1),\ldots,\FF_{nd}^*(x_d))-g(\FF_{n1}(x_1),\ldots,\FF_{nd}(x_d))  \} \, {\rm d}C(\x) +\\
&&\quad  +  \int g(\FF_{n1}^*(x_1),\ldots,\FF_{nd}^*(x_d)) -  g(\FF_{n1}(x_1),\ldots,\FF_{nd}(x_d))\, {\rm d}  (\FF_n-C)(\x)\\
&&
\quad + \int \{ g(\FF_{n1}^*(x_1),\ldots,\FF_{nd}^*(x_d)) -  g(\x) \}\, {\rm d}(\FF_n^*-\FF_n)(\x)
\end{eqnarray*}
so that
\begin{eqnarray} \label{id}
\sup_{g\in\G} \left| \bar\ZZ_n^*(g) - \widetilde \ZZ_n^*(g) \right| &\le& I^* + II^* + III^*
\end{eqnarray}
with
\begin{eqnarray*}
I^* &=&  \sup_{g\in\G}
\left| \int \left \{ \sqrt{n} \left[ g(\FF_{n1}^*(x_1),\ldots,\FF_{nd}^*(x_d))-g(\FF_{n1}(x_1),\ldots,\FF_{nd}(x_d)) \right] -
  \sum_{k=1}^d \UU_{nk}^*(x_k)   \dot{g}_k(\x) \right\} \, {\rm d}C(\x) \right|
\\
  II^* &=& \sup_{g\in\G}\left|  \int \{ g(\FF_{n1}^*(x_1),\ldots,\FF_{nd}^*(x_d)) -  g(\FF_{n1}(x_1),\ldots,\FF_{nd}(x_d))\} \, {\rm d}  \UU_n(\x) \right|\\
  III^* &=& \sup_{g\in\G} \left| \int \{ g(\FF_{n1}^*(x_1),\ldots,\FF_{nd}^*(x_d) ) - g(\x) \}\,{\rm d} \UU_n^* (\x) \right|
  \end{eqnarray*}
For the first term on the right in (\ref{id}), we reason as in Proposition 6 and we use that $\dot{g}_k$ is uniformly equicontinuous and both $\UU_{nk}$ and  $\UU_{nk}^*$
converge weakly. More precisely, there exists a bounded function $\phi_k$ with $\lim_{t\downarrow0}
 \phi_k(t)=0$ such that
\begin{eqnarray*}
&& \sup_{g\in\G}
\left| \int \left \{\sqrt{n}\left[ g(\FF_{n1}^*(x_1),\ldots,\FF_{nd}^*(x_d))-g(\FF_{n1}(x_1),\ldots,\FF_{nd}(x_d)) \right]-
  \sum_{k=1}^d \UU_{nk}^*(x_k)   \dot{g}_k(\x) \right\} \, {\rm d}C(\x) \right|\\
  &\le& \sum_{k=1}^d \| \UU_{nk}^*\|_\infty \phi_k( \| \FF_n^* - \FF_n\|_\infty +\| \FF_n-I\|_\infty ) = o_{p^*}(1), \text{ as } n\to\infty,
  \end{eqnarray*}
  as
\[  \phi_k(  \| \FF_n^* - \FF_n \|_\infty +  \| \FF_n - I \|_\infty ) =o_{p^*}(1), \text{ as } n\to\infty.
\]
For the second term on the right in (\ref{id}), we write
\begin{eqnarray*}
&& \sup_{g\in\G}\left|  \int \{ g(\FF_{n1}^*(x_1),\ldots,\FF_{nd}^*(x_d)) -  g(\FF_{n1}(x_1),\ldots,\FF_{nd}(x_d))\} \, {\rm d}  \UU_n(\x) \right|\\
&&\le  \sum_{k=1}^d \left| \int  \UU_{nk}^* (x_k)  \dot{g}_k(\x)  \, {\rm d}  n^{-1/2}\UU_n(\x)  \right|  + 2\sum_{k=1}^d  \|  \UU_{nk}^*\|_\infty\phi_k(
 \| \FF_n^* - \FF_n\|_\infty + \| \FF_n-I\|_\infty )   \\
 &&=\sum_{k=1}^d \left| \int  \UU_{nk}^* (x_k)  \dot{g}_k(\x)  \, {\rm d}  n^{-1/2}\UU_n(\x)  \right|   +o_{p^*}(1)
\end{eqnarray*}
for the bounded functions $\phi_k$ with $\lim_{t\downarrow 0} \phi_k(t)=0$. Moreover, for each $\eps>0$ there exists a $M=M(\eps)<\infty$ such that  the event $$\{ \| n^{-1/2} \UU_{nk}^*\|_\infty\le M/2\} \cap \{\| n^{-1/2} \UU_{nk}\|_\infty\le M/2\}$$ has probability at least $1-\eps$ for $n\to\infty$, and on this event
\begin{eqnarray*}
\sup_{g\in\G} \left | \int \UU_{nk}^*(x_k) \dot{g}_k(\x) \, {\rm d} n^{-1/2} \UU_n(\x) \right| &\le& \sup_{h\in \calH_{k,n}(M)} \left| \int h \, {\rm d} n^{-1/2} \UU_n \right|
=o_p(1)
\end{eqnarray*} as $n\to\infty$,
by the same reasoning as in Proposition 7,
with the class $\calH_{n,k}(M)$ defined in (\ref{H_n}). Note that, by the triangle inequality, $\dot{g}_k\UU_{nk}^*\in \calH_{n,k}(M)$ on the above event.\\

For the third term on the right in (\ref{id}), we can argue as for the previous  term above, now using the weak convergence of $\UU_n^*$ in lieu of $\UU_n$.
In particular, for each fixed $\eps>0$, choose $M<\infty$  for which
$$ \bigcap_{k=1}^d \left\{ |\sqrt{n}( \FF_{nk}^* -I) \|_\infty\le M\right\}, $$
holds with (bootstrap) probability at least $1-\eps$, as $n\to\infty$. On this event,  $\sqrt{n}( \FF_{nk}^*-I)\dot{g}_k $ belongs to $\calH_{n,k}(M)$, $k=1,\ldots,d$,
 and
\begin{eqnarray*}
&&\sup_{g\in\G} \left| \int \{ g(\FF_{n1}^*(x_1),\ldots,\FF_{nd}^*(x_d) ) - g(\x) \}\,{\rm d} \UU_n^* (\x) \right|\\
&&\le \sum_{k=1}^d \sup_{h\in \calH_n,k(M)}\left| \int h \, {\rm d} n^{-1/2} \UU_n^*\right| +2\sum_{k=1}^d  \phi_k(\| \FF_{n}^* - I\|_\infty) \| \sqrt{n} (\FF_{nk}^* -I) \|_\infty\\
&&= \sum_{k=1}^d \sup_{h\in \calH_n,k(M)}\left| \int h \, {\rm d} n^{-1/2} \UU_n^*\right| +o (1), \text{ as } n\to\infty.
\end{eqnarray*}
The proof  of Theorem~2.4.3 of Van der Vaart \& Wellner (1996) or the proof of  the uniform Glivenko-Cantelli theorem (Theorem~2.8.1 of Van der Vaart \& Wellner (1996)) show that
\[ \sup_{h\in \calH_{n,k}(M)}\left| \int h \, {\rm d} n^{-1/2} \UU_n^*\right| =o_{p^*}(1) \text{ as } n\to\infty
\]
as all functions $h\in\calH_{n,k}(M)$ are uniformly bounded and the required entropy condition is met with ease, as (\ref{eq:entropy_intermediate_1}) shows that
\[ \sup_Q \log N(\xi,\calH_{k,n}(M), L_1(Q)) = O(\sqrt{n}),
\]
with the supremum taken over all probability measures $Q$, for all $\xi>0$.\\

Hence (\ref{eq}) holds, and the theorem follows from the weak convergence of $\widetilde \ZZ_n^*$.\qed

\bigskip

\subsection{Technical results}
\label{tr}

This subsection contains technical lemmata needed for the proof of Theorem~\ref{theorem:ecp2}.

\begin{lemma}
\label{lemma:uniform_equicontinuity_function_phi}
Let   $\calF$  be the  class of  uniformly equicontinuous functions $f:[0,1]^d\to\RR$.  Then there exists a monotone  increasing function $\phi_{\calF}:\mathbb{R}^+\rightarrow\mathbb{R}^+$ such that
\begin{align}
\lim_{x\downarrow0}\phi_{\calF}(x)=0
\label{eq:phi_calG_zero_limit}
\end{align}
and
\begin{align}
 \sup_{f\in\calF}|f(\x)-f(\y)|\le\phi_\F(\|\x-\y\|) \ \text{ for all } \x,\y\in[0,1]^d.
\label{eq:phi_calG_property}
\end{align}
In addition, 
$\phi_{\calF}$ is finite valued.
\end{lemma}

\begin{proof}
Let
\begin{align}
\phi_{\calF}(a) = \sup_{\|\x-\y\|\le a}\sup_{f\in\calF}|f(\x)-f(\y)|. \nonumber
\end{align}
Clearly, $\phi_\F$ is monotone increasing.  In addition,  (\ref{eq:phi_calG_zero_limit}) must hold, for otherwise $\calF$ is not uniformly equicontinuous.  Next, for any $\x',\y'$, we let $\delta=\|\x'-\y'\|$, and we observe that
\begin{align}
\sup_{f\in\calF}|f(\x')-f(\y')| \le \sup_{\|\x-\y\|\le\delta}\sup_{f\in\calF}|f(\x)-f(\y)|=\phi_{\calF}(\delta)=\phi_\F(\|\x'-\y'\|). \nonumber
\end{align}
It remains to show that $\phi_{\calF}$ is finite valued when $\calF$ is defined on   the bounded set $[0,1]^d$.  By  (\ref{eq:phi_calG_zero_limit}), we can choose $\delta \in(0,a]$ small enough such that $\phi_{\calF}(\delta)<\infty$.  For each pair of $\x,\y\in[0,1]^d$, we construct a $\delta$-chain
\begin{align}
\{\x=\x_{\x,\y,0},\x_{\x,\y,1},\dots,\x_{\x,\y,k_{\x,\y}}=\y\} \nonumber
\end{align}
such that
\begin{align}
\|\x_{\x,\y,i}-\x_{\x,\y,i-1}\|\le\delta \nonumber
\end{align}
and $k_{\x,\y} \le C/\delta$ for some finite constant $C$ not dependent on $\x,\y$.  Note that this choice of $k_{\x,\y}$ with the given specific bound is possible because the class $\calF$ is defined on a bounded set.  Then, by the construction of $\phi_{\calF}$, we have
\begin{align}
\phi_{\calF}(a) &= \sup_{\|\x-\y\|\le a}\sup_{f\in\calF}|f(\x)-f(\y)| \nonumber \\
&\le \sup_{\|\x-\y\|\le a}\sup_{f\in\calF} \sum_{i=1}^{k_{\x,\y}}|f(\x_{\x,\y,i})-f(\x_{\x,\y,i-1})| \nonumber \\
&\le \sup_{\|\x-\y\|\le a} \sum_{i=1}^{k_{\x,\y}} \phi(\delta) \nonumber \\
&\le (C/\delta) \phi(\delta)
<\infty. \nonumber
\end{align}
\end{proof}

\bigskip
\begin{lemma}
\label{lemma:entropy_uniformly_equicontinuity_calF_k}
The class $\F$ of uniformly bounded and uniformly equicontinuous functions $f:[0,1]^d \to[0,1]$ is totally bounded in $L_\infty([0,1]^d)$.
\end{lemma}
\begin{proof}
Since there exists a finite, monotone increasing  function $\phi$ with $\lim_{t\downarrow0}\phi(t)=0$  such that
$\sup_{f\in\F}  |f(\x)-f(\y)| <\phi(\|\x-\y\|)$ for all $\x,\y\in[0,1]^d$ by Lemma 8, and since the domain $[0,1]^d$  is bounded, we can construct,  for each $\eps>0$,  a regular $\delta$-grid of $[0,1]^d$  with
$\delta=\inf\{ t>0:\  \phi(t\sqrt{d} )\ge \eps/2\}$  strictly positive. Since $\F$ is uniformly bounded, it is easy to see that using this finite grid with $\delta$ given above, there are  finitely many functions  $g_1,\ldots,g_M$ such that
$ \min_{1\le k\le M}  \sup_{f\in\F}  \|f-g_k\|_\infty<\varepsilon $.
(Using the line of reasoning as in the proof of Lemma 2.3 in Van de Geer (2000), we can actually improve this crude bound).
\end{proof}

\bigskip

\begin{lemma}
\label{mono}
 The class $\F$ of monotone functions $f:\RR\to[0,1]$ satisfies
 \[ \log N(\eps,\F,L_r(Q))\le K \frac 1\eps
 \]
 for all $\eps>0$, all probability measures $Q$ and all $r\ge1$.
 \end{lemma}
 \begin{proof} See Theorem~2.7.5 of Van der Vaart \& Wellner (1996).
 \end{proof}

 \bigskip

\appendix

\section{Integration by parts}

In this section we present a general yet simple integration by parts formula in arbitrary dimensions.  The integration by parts formula in one-dimension is well-known; see for instance, Saks (1937, Theorem~14.1 in Chapter~3).  However, its multivariate extension does not appear to be adequately addressed.  To the best of our knowledge, the only general integration by parts formula in arbitrary dimensions is Proposition A.1 of Fermanian (1998, page 149).  However, this formula is still somewhat complicated; for instance, the resulting integrand is expressed in terms of the measures defined from the functions, rather than the functions themselves.  As we will see, we will derive a significantly simplified and essentially optimal formula.  Recently, Theorem~A.6 and Corollary~A.7 of Berghaus et al. (2014) derive a simple bivariate integration by parts formula; however, their proof is quite specific to the case $d=2$ and does not appear to be easily generalizable to higher dimensions.  We will demonstrate how we recover Corollary~A.7 of Berghaus et al. (2014) from our general formula.  We refer the readers to Section~\ref{sec:notations} for notations, which in turn follow Owen~(2005).  To obtain a more general formula we consider the domain $[\a,\b]$ for $\a<\b$ instead of the unit hypercube $[{\bf 0},\1]$ on which we focused for the remainder of the paper.


\begin{theorem}[Integration by parts]\label{IBP-main}
We let $\a,\b\in\mathbb{R}^d$, and functions $f, g: [\a,\b]\rightarrow\mathbb{R}$.  We assume that $f, g$ both satisfy assumption~F with $[0,1]^d$ replaced by $[\a,\b]$.  Then we have the integration by parts formula:
\begin{align}
&\int_{(\a,\b]} f(\x) {\rm d} g(\x) \nonumber \\
&= \sum_{ \substack{I_1,I_2,I_3\subset\{1,\dots,d\}:\\ I_1+I_2+I_3 = \{1,\dots,d \}} } (-1)^{|I_1|+| I_2|} \int_{(\a_{I_1},\b_{I_1}]} g(\x_{I_1}-;\a_{I_2}:\b_{I_3}) \, \d f(\x_{I_1};\a_{I_2}:\b_{I_3}).
\label{eq:IBP}
\end{align}
In the above the $+$ symbol within $I_1+I_2+I_3$ denotes the disjoint union, so the summation is taken over all partitions of the set $\{1,\ldots,d\}$ into the sets $I_1, I_2, I_3$.
\end{theorem}

\subsection*{Remarks.}
\begin{itemize}
\item
If either $f$ or $g$ is continuous, then we can replace $\x_{I_1}-$ in (\ref{eq:IBP}) by $\x_{I_1}$.  \\

\item
In Theorem~\ref{IBP-main} we used the convention that if $I_1=\emptyset$, then $$\int_{(\a_{I_1},\b_{I_1}]} g(\x_{I_1}-;\a_{I_2}:\b_{I_3}) \d f(\x_{I_1};\a_{I_2}:\b_{I_3}) = g(\a_{I_2}:\b_{I_3}) f(\a_{I_2}:\b_{I_3}) = (fg)(\a_{I_2}:\b_{I_3}),$$
that is, when there is no integration variable, the integration operation simply disappears.\footnote{In Owen (2005), the constant $f(c)$ on the right hand side of Equation~(13) is alternatively written as $\Delta_{u=\emptyset}(f;x,c)$, which appears in the first line of the equation array in the proof that follows and which resembles a measure.  We follow suit and pretend that the constant $f(c)$ in fact gives rise to a measure.}  This convention allows for a more condensed expression, and we will use the same convention below.
Alternatively, to avoid invoking this convention, (\ref{eq:IBP}) can be equivalently expressed as
\begin{align}
&\int_{(\a,\b]} f(\x)\, \d g(\x) \nonumber \\
&= \sum_{ \substack{I_1,I_2,I_3\subset\{1,\dots,d\}:\\ I_1\neq\emptyset,I_1+I_2+I_3 = \{1,\dots,d\}} } (-1)^{|I_1|+|I_2|} \int_{(\a_{I_1},\b_{I_1}]} g(\x_{I_1}-;\a_{I_2}:\b_{I_3}) \, \d f(\x_{I_1};\a_{I_2}:\b_{I_3}) \nonumber \\
&\ + \sum_{ I\subset\{1,\dots,d\} } (-1)^{|I |} (fg)(\a_{I}:\b_{-I}).
\label{eq:IBP_alt}
\end{align}
Note that the summation in the last line of (\ref{eq:IBP_alt}) is just $(fg)((\a,\b])$, the weight of the cube $(\a,\b]$ assigned by the measure defined from the product $(fg)$.

\item
For $d=1$, the formula reduces to
\begin{eqnarray*}
\int_{(a,b]} f(x) \d g(x)  &=& - \int_{(a,b]} g(x-) \d f(x) + \left[ (fg)(b) - (fg)(a) \right] \\
&=&  - \int_{(a,b]} g(x) \d f(x) + \left[ (fg)(b) - (fg)(a) \right] +  \int_{(a,b]}\{ g(x)-g(x-) \} \d f(x).
\end{eqnarray*}
The second term on the right is the generalized  length of the interval $(a,b] $, based on the right-continuous function $(fg)$, and the third term equals $\sum_{x} \{ g(x)-g(x-) \} \{ f(x)-f(x-) \} $ with the summation taken  over all (countable) common points $x$ of  discontinuity of $f$ and $g$. This term vanishes  if either $f$ or $g$ is continuous.
\item
Note that no further simplification of (\ref{eq:IBP}) is possible, as clearly all terms in the sum, now expressed directly in terms of the functions $f$ and $g$ (instead of the measures they define), are distinct.  In contrast, Proposition~A.1 of Fermanian (1998) still contains duplicates that are hidden in the measures that appear as integrands.  In $d$ dimensions, there are a total of $3^d$ terms in the sum of (\ref{eq:IBP}), because each of the $d$ coordinates belongs to exactly one of the sets $I_1, I_2, I_3$.
If $d=2$ and if either $f$ or $g$ is continuous, these $3^2=9$ terms exactly correspond to Corollary~A.7 of Berghaus et al. (2014), as we demonstrate now.  In this case, (\ref{eq:IBP_alt}) (which is equivalent to (\ref{eq:IBP}) with our convention) claims
\begin{align}
&\int_{(\a,\b]} f(\x) \d g(\x) \nonumber \\
&= \sum_{ \substack{I_1,I_2,I_3\subset\{1,2\}:\\ I_1\neq\emptyset, I_1+I_2+I_3 = \{1,2\}} } (-1)^{|I_1|+|I_2|} \int_{(\a_{I_1},\b_{I_1}]} g(\x_{I_1};\a_{I_2}:\b_{I_3}) \d f(\x_{I_1};\a_{I_2}:\b_{I_3}) \nonumber \\
&\ + \sum_{ I\subset\{1,2\} } (-1)^{|I |} (fg)(\a_{I}:\b_{-I}).
\label{eq:IBP_alt_d_2}
\end{align}

We work out the terms on the right hand side of (\ref{eq:IBP_alt_d_2}) one by one.  We start from the second line on the right hand side (\ref{eq:IBP_alt_d_2}).  For $I=\emptyset$, we have
\begin{align}
(-1)^{0} (fg)(\a_{\emptyset}:\b_{\{1,2\}}) = (fg)(b_1,b_2). \nonumber
\end{align}
For $I=\{1\}$, we have
\begin{align}
(-1)^{1} (fg)(\a_{\{1\}}:\b_{\{2\}}) = -(fg)(a_1,b_2). \nonumber
\end{align}
For $I=\{2\}$, we have
\begin{align}
(-1)^{1} (fg)(\a_{\{2\}}:\b_{\{1\}}) = -(fg)(b_1,a_2). \nonumber
\end{align}
For $I=\{1,2\}$, we have
\begin{align}
(-1)^{2} (fg)(\a_{\{1,2\}}:\b_{\emptyset}) = (fg)(a_1,a_2). \nonumber
\end{align}

We now switch to the first line on the right hand side (\ref{eq:IBP_alt_d_2}).  For $I_1=\{1\}$, $I_2=\emptyset$, and so $I_3=\{2\}$, we have
\begin{align}
 (-1)^{1+0} \int_{(\a_{\{1\}},\b_{\{1\}}]} g(\x_{\{1\}};\b_{\{2\}}) \d f(\x_{\{1\}};\b_{\{2\}}) = - \int_{(a_1,b_1]} g(x_1,b_2) f(\d x_1,b_2). \nonumber
\end{align}
For $I_1=\{1\}$, $I_2=\{2\}$, and so $I_3=\emptyset$, we have
\begin{align}
(-1)^{1+1} \int_{(\a_{\{1\}},\b_{\{1\}}]} g(\x_{\{1\}};\a_{\{2\}}) \d f(\x_{\{1\}};\a_{\{2\}}) = \int_{(a_1,b_1]} g(x_1,a_2) f(\d x_1,a_2). \nonumber
\end{align}
For $I_1=\{2\}$, $I_2=\emptyset$, and so $I_3=\{1\}$, we have
\begin{align}
 (-1)^{1+0} \int_{(\a_{\{2\}},\b_{\{2\}}]} g(\x_{\{2\}};\b_{\{1\}}) \d f(\x_{\{2\}};\b_{\{1\}}) = - \int_{(a_2,b_2]} g(b_1,x_2) f(b_1,\d x_2). \nonumber
\end{align}
For $I_1=\{2\}$, $I_2=\{1\}$, and so $I_3=\emptyset$, we have
\begin{align}
(-1)^{1+1} \int_{(\a_{\{2\}},\b_{\{2\}}]} g(\x_{\{2\}};\a_{\{1\}}) \d f(\x_{\{2\}};\a_{\{1\}}) = \int_{(a_2,b_2]} g(a_1,x_2) f(a_1,\d x_2). \nonumber
\end{align}

Finally, for $I_1=\{1,2\}$, $I_2=\emptyset$, and so $I_3=\emptyset$, we have
\begin{align}
(-1)^{2+0} \int_{(\a_{\{1,2\}},\b_{\{1,2\}}]} g(\x_{\{1,2\}}) \d f(\x_{\{1,2\}}) = \int_{(\a,\b]} g(x_1,x_2) \d f(x_1,x_2) = \int_{(\a,\b]} g \d f. \nonumber
\end{align}

Thus in the end, after collecting all terms and replacing the dummy integration variables $x_1$ and $x_2$ by $u$ and $v$, we have
\begin{align}
&\int_{(\a,\b]} f(\x) \d g(\x) \nonumber \\
&= \int_{(\a,\b]} g \d f + (fg)(b_1,b_2) - (fg)(a_1,b_2) - (fg)(b_1,a_2) + (fg)(a_1,a_2) \nonumber \\
& - \int_{(a_1,b_1]} g(u,b_2) f(\d u,b_2) + \int_{(a_1,b_1]} g(u,a_2) f(\d u,a_2) \nonumber \\
& - \int_{(a_2,b_2]} g(b_1,v) f(b_1,\d v) + \int_{(a_2,b_2]} g(a_1, v) f(a_1,\d v). \nonumber
\end{align}
The terms in the above equation are arranged as and correspond exactly to the terms in Corollary~A.7 of Berghaus et al (2014) (note that their $\Delta$ term is exactly the sum of the second to the fifth terms on the right hand side of the above equation).

\item
If $g$ is the empirical copula process so $[\a,\b]=[0,1]^d$, then (\ref{eq:IBP}) can be further simplified by setting $I_2=\emptyset$, as the function $g$ becomes zero as soon as some of its argument becomes zero. This leads to formula (\ref{IBPformula}) in Proposition \ref{IBP!}.
\end{itemize}

\begin{proof}[Proof of Theorem~\ref{IBP-main}]
We recall how the functions $f$ and $g$, and the lower-dimensional projections of the latter, are identified with measures satisfying (\ref{def:measure}) and (\ref{def:measure_projection}).
Using Equation~(13) (or more precisely, using the left hand side of the first line of the equation array in the proof that follows) of Owen (2005) (see also Lemma~A.2 of Fermanian~(1998)), we have
\begin{align}
f(\x) = \sum_{I\subset{\{1,\ldots,d\}}}(-1)^{|I|} f((\x_I,\b_I];\b_{-I}). \nonumber
\end{align}
Plugging this into $\int_{(\a,\b]} f(\x) \, \d g(\x)$, we have
\begin{align}
\int_{(\a,\b]} f(\x)\, \d g(\x) &= \sum_{I \subset\{1,\ldots,d\} }(-1)^{|I|} \left[ \int_{(\a,\b]} f((\x_I,\b_I];\b_{-I}) \, \d g(\x) \right] \nonumber \\
&= \sum_{I\subset\{1,\ldots,d\} }(-1)^{|I|} \left[ \int_{(\a,\b]} \int_{(\a_I,\b_I]} \1\left\{ \x_I <\y_I\le \b_I \right\} \, \d f(\y_I;\b_{-I}) \, \d g(\x) \right]. \nonumber
\end{align}
Applying Fubini's theorem, we obtain
\begin{align}
&\int_{(\a,\b]} f(\x) \, \d g(\x) \nonumber \\
&= \sum_{I\in\{1,\ldots,d\}}(-1)^{|I|} \left[ \int_{(\a_I,\b_I]} \int_{(\a,\b]} \1\left\{ \a_I<\x_I<\y_I \right\} \, \d g(\x)\, \d f(\y_I;\b_{-I}) \right] \nonumber \\
&= \sum_{I\subset\{1,\ldots,d\}}(-1)^{|I|} \left[ \int_{(\a_I,\b_I]} g((\a_I,\y_I)\times(\a_{-I},\b_{-I}]) \, \d f(\y_I;\b_{-I}) \right].
\label{eq:IBP-main_int}
\end{align}
Up to this point, we have essentially derived a variant of Proposition~A.1 of Fermanian (1998).  However, (\ref{eq:IBP-main_int}) can be significantly simplified.  Continuing from (\ref{eq:IBP-main_int}), we have
\begin{align}
&\int_{(\a,\b]} f(\x) \, \d g(\x) \nonumber \\
&= \sum_{I\subset\{1,\ldots,d\}}(-1)^{|I|} \left[ \int_{(\a_I,\b_I]} \lim_{\x_I\uparrow\y_I-} g((\a_I,\x_I]\times(\a_{-I},\b_{-I}])\, \d f(\y_I;\b_{-I}) \right] \nonumber \\
&= \sum_{I\subset\{1,\ldots,d\}}(-1)^{|I|} \left[ \int_{(\a_I,\b_I]} \lim_{\x_I\uparrow\y_I-} \sum_{I_1\subset I}\sum_{I_2\subset-I} (-1)^{|I_1|+|I_2|} g(\x_{I-I_1}:\a_{I_1+I_2}:\b_{-I-I_2}) \, \d f(\y_I;\b_{-I}) \right] \nonumber \\
&= \sum_{I\subset\{1,\ldots,d\}}\sum_{I_1\subset I}\sum_{I_2\subset-I}(-1)^{|I|+|I_1|+|I_2|} \left[ \int_{(\a_I,\b_I]} g(\y_{I-I_1}-:\a_{I_1};\a_{I_2}:\b_{-I-I_2}) \, \d f(\y_I;\b_{-I}) \right].
\label{eq:IBP_int_2}
\end{align}
The term in the square bracket in the last line of (\ref{eq:IBP_int_2}) can be further simplified as
\begin{align}
\sum_{I_3\subset I_1}(-1)^{|I_3|} \int_{(\a_{I-I_1},\b_{I-I_1}]} g(\y_{I-I_1}-;\a_{I_1+I_2}:\b_{-I-I_2}) \, \d f(\y_{I-I_1};\a_{I_3}:\b_{-I+I_1-I_3}).
\label{eq:step_fun}
\end{align}
This claim is easily verified  for step functions $g$ (in $\y_{I-I_1}$) and the general case follows by approximating $g$ by a sum of step functions.
After replacing the term in the square bracket in the last line of (\ref{eq:IBP_int_2}) by  (\ref{eq:step_fun}), we  obtain
\begin{align}
 \int_{(\a,\b]} f(\x)\, \d  g(\x) &=  \sum_{I\subset\{1,\ldots,d\}}\sum_{I_1\subset I}\sum_{I_2\subset-I}\sum_{I_3\subset I_1} (-1)^{|I|+|I_1|+|I_2|+|I_3|}\nonumber\\
 & \left[ \int_{(\a_{I-I_1},\b_{I-I_1}]} g(\x_{I-I_1}-;\a_{I_1+I_2}:\b_{-I-I_2}) \, \d f(\x_{I-I_1};\a_{I_3}:\b_{-I+I_1-I_3}) \right].
\label{eq:IBP_int_3}
\end{align}
We now simplify the summation.  We let $u_1=I-I_1,u_2=I_3,u_3=-I-I_2,u_4=I_2,u_5=I_1-I_3$, so $u_1+u_2+u_3+u_4+u_5=\{1,\ldots,d\}$. The summation over ${I\subset\{1,\ldots,d\}}$, ${I_1\subset I}$, ${I_2\subset-I}$, ${I_3\subset I_1}$    becomes a summation over ${u_1,u_2,u_3,u_4,u_5}$ with ${ u_1+u_2+u_3+u_4+u_5=\{1,\ldots,d\}},$ and
$(-1)^{|I|+|I_1|+|I_2|+|I_3|}$ becomes $(-1)^{|u_1|+3|u_2|+|u_4|+2|u_5|}=(-1)^{|u_1|+|u_2|+|u_4|}.$  Finally, (\ref{eq:IBP_int_3}) becomes
\begin{align}
\int_{(\a,\b]} f(\x) \, \d g(\x) &= \sum_{ \substack{u_1,u_2,u_3,u_4,u_5\subset\{1,\dots,d\}:\\u_1+u_2+u_3+u_4+u_5=\{1,\dots,d\}} } (-1)^{|u_1|+|u_2|+|u_4|} \nonumber \\
&\left[ \int_{(\a_{u_1},\b_{u_1}]} g(\x_{u_1}-;\a_{u_2+u_4+u_5}:\b_{u_3})\, \d f(\x_{u_1};\a_{u_2}:\b_{u_3+u_4+u_5}) \right].
\label{eq:IBP_int_4}
\end{align}
Note that in the square bracket, $u_4$ and $u_5$ always appear together as the union $u_4+u_5$, and so the term in the square bracket is uniquely determined by $u_1$, $u_2$, $u_3$ (in which case $u_4+u_5$ is also determined).  Now we evaluate, for given $u_1$, $u_2$, $u_3$, the coefficient  \begin{align}
\sum_{ \substack{u_4,u_5\subset\{1,\dots,d\}:\\u_4+u_5=\{1,\dots,d\}-u_1-u_2-u_3} } (-1)^{|u_1|+|u_2|+|u_4|} = (-1)^{|u_1|+|u_2|} \sum_{ \substack{u_4,u_5\subset\{1,\dots,d\}:\\u_4+u_5=\{1,\dots,d\}-u_1-u_2-u_3} } (-1)^{|u_4|}. \nonumber
\end{align}
One moment's thought reveals that the summation on the right hand side of the above equality is zero, unless $u_4+u_5=\emptyset$, in which case it is one.\footnote{If $u_4+u_5\neq\emptyset$ and $|u_4+u_5|$ is odd, for each $u\subset u_4+u_5$, the term $(-1)^{|u_4|}$ with $u_4=u$ cancels with the term $(-1)^{|u_4|}$ with $u_4=u_4+u_5-u$, because exactly one of $|u_4|$ and $|u_4+u_5-u|$ is odd.  If $u_4+u_5\neq\emptyset$ and $|u_4+u_5|$ is even, and if $i\in u_4+u_5$, we can separately consider the case $i\in u_4$ and $i\notin u_4$ to effectively reduce the number of coordinates to consider from even to odd, and apply the previous argument again.}
\\
Finally, (\ref{eq:IBP}) follows by applying the above lemma to (\ref{eq:IBP_int_4}).
\end{proof}

\bigskip
\section{Measure defined from lower dimensional projections of functions of bounded HK variation}

We show that for all $\c\in\{0,1\}^d$ the measure (\ref{def:measure_projection}) on $[{\bf 0}_I, \1_I]$ defined from the function $f(\cdot;\c_{-I})$, the lower dimensional projection of the function $f$ on $[{\bf 0}_I, \1_I]$, is well defined.  (In fact we can show this for all $\c\in[{\bf 0},\1]$.)  Because $f(\cdot;\c_{-I})$ is obviously right-continuous, it suffices to show that $f(\cdot;\c_{-I})$ is of bounded HK variation on $[{\bf 0}_I, \1_I]$.  By definition of the HK variation, it in turn suffices to show that the Vitali variation of the function $f(\cdot;\1_{I-I'}:\c_{-I})$ on $[{\bf 0}_{I'}, \1_{I'}]$ is finite for all $I'\subset I$ with $I'\neq\emptyset$ (if $I'=\emptyset$ then $f(\cdot;\1_{I-I'}:\c_{-I})$ is just the constant that we added on top of the definition of the original HK variation).

We essentially proceed as in the proof of Lemma~2 (up to roughly the middle of page~17) of Aistleitner~\&~Dick~(2014).  We let $f(\x)=f({\bf 0})+f^+(\x)+f^-(\x)$ be the Jordan decomposition of $f$, so that both $f^+$ and $f^-$ are \textit{completely monotone} in the sense defined at the top of page~11 of Aistleitner~\&~Dick~(2014).  We fix arbitrary $I'\subset I$ with $I'\neq\emptyset$.  By closure of the property of bounded Vitali variation under summation, it suffices to show that the Vitali variations of the functions $f^{\pm}(\cdot;\1_{I-I'}:\c_{-I})$ on $[{\bf 0}_{I'}, \1_{I'}]$ are finite.  Without loss of generality we show this for $f^+(\cdot;\1_{I-I'}:\c_{-I})$.  Because $f^+(\cdot;\1_{I-I'}:\c_{-I})$ is completely monotone, its Vitali variation on $[{\bf 0}_{I'}, \1_{I'}]$ is simply
\begin{align}
\sum_{I''\in I'}(-1)^{|I''|} f^+({\bf 0}_{I''}:\1_{I-I''}:\c_{-I}). \nonumber
\end{align}
By the correspondence of $f^+$ and the measure $\nu^+$, as defined  in Aistleitner~\&~Dick~(2014), the above equals
\begin{align}
\nu^+(\left\{\x\in[{\bf 0},\1]:{\bf 0}_{I'}<\x_{I'}\le\1_{I'}, \x_{I-I'}\le\1_{I-I'}, \x_{-I}\le\c_{-I}\right\}). \nonumber
\end{align}
Since $I'\neq\emptyset$ the set in the parenthesis above is a subset of $[{\bf 0}, \1]-\{{\bf 0}\}$, and so
\begin{align}
\nu^+(\left\{\x\in[{\bf 0},\1]:{\bf 0}_{I'}<\x_{I'}\le\1_{I'}, \x_{I-I'}\le\1_{I-I'}, \x_{-I}\le\c_{-I}\right\})\le \nu^+([{\bf 0}, \1]-\{{\bf 0}\}) = f^+(\1) <\infty, \nonumber
\end{align}
and we are done.\\

\section*{Acknowledgement} The research of Wegkamp is supported in part by the NSF DMS 1310119 grant. \\


\begin{thebibliography}{99}


\bibitem{} C. Aistleitner and J. Dick (2014). Functions of bounded variation, signed measures, and a general Koksma-Hlawka inequality.  {\em Acta Arith.}, {\bf 167}, 143--171, 2015.



\bibitem{}
 D.W.K.~Andrews and D.P.~Pollard (1994). An Introduction to Functional Central Limit Theorems for Dependent Stochastic Processes. {\em International Statistical Review}, \textbf{62}(1), 119--132.
 
 \bibitem{}
M.A.~Arcones,  and B.~Yu (1994). Central limit theorems for empirical and U-processes of stationary mixing sequences. {\em Journal of Theoretical Probability} \textbf{7}, 47--71.
 

\bibitem{} B. Berghaus, A. B\"ucher, and S. Volgushev. (2014). Weak convergence of the empirical copula process with respect to weighted metrics. 	arXiv:1411.5888 [math.ST]


\bibitem{} A. B\"ucher, J. Segers,  and S. Volgushev (2014).
When uniform weak convergence fails: Empirical processes for dependence functions and residuals via epi- and hypographs"
{\em Annals of Statistics}, {\bf 42}(4), 
1598--1634".


\bibitem{}
A. B\"ucher, and S. Volgushev. (2013). Empirical and sequential empirical copula processes under serial dependence.
{\em Journal of Multivariate Analysis}, 119, 61--70.



\bibitem{deheuv1} {P. Deheuvels} (1979). La fonction de
d\'ependance empirique et ses propri\'et\'es. {\em Acad. Roy.
Belg., Bull. C1 Sci. 5i\`{e}me s\'er.}, {\bf 65}, 274--292.

 \bibitem{} H. Dehling, O. Durieu and D. Volny (2009). New techniques for empirical processes of dependent data. {\em Stochastic Processes and their Applications}, \textbf{119}, 3699--3718.
 
 
 \bibitem{} 
H. Dehling, O. Durieu and M. Tusche (2014). Approximating class approach for empirical processes of dependent sequences indexed by functions. {\em Bernoulli}, \textbf{20}, 1372--1403.
 
 \bibitem{} P. Doukhan, and D. Surgailis (1998). Functional central limit theorem for the empirical process of short memory linear processes. {\em Comptes Rendus de l'Acad\'emie des Sciences, Math\'ematique}, \textbf{326}, 87--92.


\bibitem{} J.-D. Fermanian (1998). Contributions \`a l'Analyse Nonparam\'etrique des Fonctions de Hasard sur Donn\'ees Multivari\'ees et Censur\'ees. {\em PhD. thesis}.



\bibitem{FRW} {J.-D.\,Fermanian, D.\,Radulovi\'c and M.H.\,Wegkamp} (2004). Weak convergence of empirical copula processes. {\em Bernoulli} {\bf
10}, 847--860.

\bibitem{FRW1} {J.-D.\,Fermanian, D.\,Radulovi\'c and M.H.\,Wegkamp} (2014). Asymptotic Total Variation Tests for Copulas. {\em Bernoulli}  (in press).

\bibitem{}
{P. G\"anssler  and W. Stute} (1987).  {\em Seminar on
Empirical Processes.} DMV Seminar, Band 9, Birkh\"auser.


\bibitem{}
G. H. Hardy (1905). On double Fourier series, and especially those which
represent the double zeta-function with real and incommensurable parameters.
{\em Quarterly Journal of Mathematics} 37, 53--79.


\bibitem{hildebr}
{T.H. Hildebrandt} (1963). {\em Introduction to the Theory of Integration}. Academic
Press, New York, London.



\bibitem{}
M. Krause. (1903a). Uber mittelwerts\"atze im gebiete der doppelsummen and
doppelintegrale. {\em Leipziger Ber.} 55, 239--263.

\bibitem{}
M. Krause. (1903b). Uber Fouriersche reihen mit zwei ver"anderlichen gr\"ossen. {\em
Leipziger Ber.} 55, 164--197.


\bibitem{nelsen} {R. B. Nelsen} (1999). {\em An introduction
to copulas}, Lecture Notes in Statistics, {\bf 139}, Springer.

\bibitem{} {A. B. Owen}  (2005).
{\em Multidimensional variation for quasi-{Monte Carlo}}. In {Jianqing Fan and Gang Li}, Editors,
{International Conference on Statistics in honour of
Professor Kai-Tai Fang's 65th birthday},
 {49--74}.

\bibitem{}
 E. Rio (1998). Processus empiriques absolument r\'eguillets et entropie universelle. {\em Probability Thoery and Related Fields}, \textbf{111}, 585--608.
 
 
\bibitem{} {E. Rio} (2000). {\em Th\'{e}orie asymptotique des processus al\'{e}atoires faiblement d\'{e}pendants}, Springer Verlag.

\bibitem{} {L. R\"uschendorf} (1976). Asymptotic normality of multivariate rank order statistics. {\em
Ann. Statist.}, {\bf 4}, 912 -- 923.


\bibitem{Ruym72} {F.H.  Ruymgaart} (1973). {\em Asymptotic Theory for Rank Tests for Independence}, MC-tract 43, Mathematisch Instituut,
Amsterdam.



\bibitem{} {F.H. Ruymgaart} (1974).      Asymptotic normality
 of nonparametric tests for independence. {\em Ann. Statist.}, {\bf 2}, 892 -- 910.

\bibitem{} {F.H. Ruymgaart, G.R. Shorack  and W.R. van Zwet} (1972). Asymptotic normality
 of nonparametric tests for independence. {\em Ann. Math. Statist.} {\bf 43}, 1122 -- 1135.


\bibitem{saks} {S. Saks} (1937). {\em Theory of the Integral}, Hafner Publishing Company.


\bibitem{segers} {J.\,Segers} (2012). Asymptotic of empirical copula processes under
nonrestrictive smoothness assumptions.  {\em Bernoulli} {\bf 18}, 764--782.

\bibitem{} W. Stute (1984). The Oscillation Behavior of Empirical Processes: The Multivariate Case. {\em The Annals of Probability} {\bf  12}, 361--379.
 \bibitem{} S.A.~van~de~Geer. (2000). {\em Empirical Processes in $M$-estimation}, Cambridge University Press.

\bibitem{} H. Tsukahara. (2005). Semiparametric estimation in copula models. {\em Canad. J. Statist. } {\bf 33}, 357--375

\bibitem{vandvw} {A.W. van der   Vaart and J.A. Wellner}
(1996). {\em Weak convergence and empirical processes}, Springer.

\bibitem{vandvw07} {A.W. van der   Vaart and J.A. Wellner}
(2007). Empirical processes indexed by estimated functions. Asymptotics: Particles, Processes and Inverse Problems, 234--252, Institute of Mathematical Statistics, Beachwood, Ohio, USA.
\end{thebibliography}
\end{document}